\documentclass[11pt]{amsart}
\usepackage[foot]{amsaddr}
\usepackage[english]{babel}

\usepackage{enumerate}

\allowdisplaybreaks

\usepackage[
pdftex,hyperfootnotes]{hyperref}
\hypersetup{
	colorlinks=true,
	linkcolor=NavyBlue, 
	urlcolor=RoyalPurple,
	citecolor=OliveGreen,
	pdftitle={Classical multiple orthogonal polynomials for arbitrary number of weights: explicit representation},
	bookmarks=true,
}
\usepackage[usenames,dvipsnames,svgnames,table,x11names]{xcolor}
\usepackage{pgfplots}
\usepackage{tikz}
\usepackage{tikz-3dplot}
\usetikzlibrary{automata,quotes, chains,matrix,calc,shadows,shapes.callouts,shapes.geometric,shapes.misc,positioning,patterns,decorations.shapes,
	decorations.pathmorphing,decorations.markings,decorations.fractals,decorations.pathreplacing,shadings,fadings,arrows.meta,bending}

\usepackage{nicematrix}
\usepackage[utf8]{inputenc}

\usepackage{comment}
\usepackage[textwidth=17.5cm,textheight=22.275cm,
height=
24.275cm,
width=18cm]{geometry}

\usepackage{amssymb,latexsym,amsmath,amsthm,bm}
\usepackage{mathrsfs}
\usepackage{mathtools,arydshln,mathdots}
\mathtoolsset{showonlyrefs}

\usepackage{tcolorbox}

\usepackage[table]{xcolor}

\usepackage{Baskervaldx}
\usepackage[]{newtxmath}

\usepackage{pifont}
\newcommand{\cmark}{\ding{51}}

\usepackage{bigints}

\theoremstyle{plain}

\newtheorem{teo}{Theorem}[section]

\renewcommand{\d}{\operatorname{d}}
\newcommand{\Exp}[1]{\operatorname{e}^{#1}}

\newcommand{\C}{\mathbb{C}}

\newcommand{\B}{\mathbb{B}}

\newcommand{\N}{\mathbb{N}}

\allowdisplaybreaks[1]

\makeatletter
\DeclareRobustCommand{\gaussk}{\DOTSB\gaussk@\slimits@}
\newcommand{\gaussk@}{\mathop{\vphantom{\sum}\mathpalette\bigcal@{K}}}

\newcommand{\bigcal@}[2]{%
	\vcenter{\m@th
		\sbox\z@{$#1\sum$}%
		\dimen@=\dimexpr\ht\z@+\dp\z@
		\hbox{\resizebox{!}{0.8\dimen@}{$\mathcal{K}$}}%
	}%
}
\newcommand{\cfracplus}{\mathbin{\cfracplus@}}
\newcommand{\cfracplus@}{%
	\sbox\z@{$\dfrac{1}{1}$}%
	\sbox\tw@{$+$}%
	\raisebox{\dimexpr\dp\tw@-\dp\z@\relax}{$+$}%
}
\newcommand{\cfracdots}{\mathord{\cfracdots@}}
\newcommand{\cfracdots@}{%
	\sbox\z@{$\dfrac{1}{1}$}%
	\sbox\tw@{$+$}%
	\raisebox{\dimexpr\dp\tw@-\dp\z@\relax}{$\cdots$}%
}
\makeatother

\makeatletter
\newcommand*{\relrelbarsep}{.386ex}
\newcommand*{\relrelbar}{%
	\mathrel{%
		\mathpalette\@relrelbar\relrelbarsep
	}%
}
\newcommand*{\@relrelbar}[2]{%
	\raise#2\hbox to 0pt{$\m@th#1\relbar$\hss}%
	\lower#2\hbox{$\m@th#1\relbar$}%
}
\providecommand*{\rightrightarrowsfill@}{%
	\arrowfill@\relrelbar\relrelbar\rightrightarrows
}
\providecommand*{\leftleftarrowsfill@}{%
	\arrowfill@\leftleftarrows\relrelbar\relrelbar
}
\providecommand*{\xrightrightarrows}[2][]{%
	\ext@arrow 0359\rightrightarrowsfill@{#1}{#2}%
}
\providecommand*{\xleftleftarrows}[2][]{%
	\ext@arrow 3095\leftleftarrowsfill@{#1}{#2}%
}
\makeatother

\newcommand*\pFqskip{8mu}
\catcode`,\active
\newcommand*\pFq{\begingroup
 \catcode`\,\active
 \def ,{\mskip\pFqskip\relax}%
 \dopFq
}
\catcode`\,12
\def\dopFq#1#2#3#4#5{%
 {}_{#1}F_{#2}\biggl[\genfrac..{0pt}{}{#3}{#4};#5\biggr]%
 \endgroup
}

\usepackage{appendix}

\newcommand{\KF}[5]{F^{#1}_{#2}\left[{#3\atop #4}\Bigg\vert #5\right]}

\usepackage{xcolor} 

\usepackage[normalem]{ulem}
\usepackage{soul} 

\usepackage{comment} 

\usepackage{tikz}
\usetikzlibrary{shapes,arrows}
\usepackage{verbatim}

\tikzstyle{block} = [draw, rectangle, 
 minimum height=3em, minimum width=2em]

\usepackage{mathrsfs} 

\newtheorem{defi}{Definición}[section]
\newtheorem{lema}[defi]{Lemma}

\newtheorem{cor}[defi]{Corollary}

\theoremstyle{remark}
\newtheorem{rem}[defi]{Remark}

\begin{document}
	\title[Classical multiple orthogonal polynomials with arbitrary number of weights]{Classical multiple orthogonal polynomials for arbitrary number of weights and their explicit representation}
\author[A Branquinho]{Amílcar Branquinho$^{1}$}
\address{$^1$CMUC, Departamento de Matemática,
	Universidade de Coimbra, 3001-454 Coimbra, Portugal}
\email{$^1$ajplb@mat.uc.pt}

\author[JEF Díaz]{Juan EF Díaz$^{2}$}
\address{$^2$CIDMA, Departamento de Matemática, Universidade de Aveiro, 3810-193 Aveiro, Portugal}
\email{$^2$juan.enri@ua.pt}

\author[A Foulquié]{Ana Foulquié-Moreno$^{3}$}
\address{$^3$CIDMA, Departamento de Matemática, Universidade de Aveiro, 3810-193 Aveiro, Portugal}
\email{$^3$foulquie@ua.pt}

\author[M Mañas]{Manuel Mañas$^{4}$}
\address{$^4$Departamento de Física Teórica, Universidad Complutense de Madrid, Plaza Ciencias 1, 28040-Madrid, Spain 
}
\email{$^4$manuel.manas@ucm.es}



\begin{abstract}
This paper delves into classical multiple orthogonal polynomials with an arbitrary number of weights, including Jacobi–Piñeiro, Laguerre of both first and second kinds, as well as multiple orthogonal Hermite polynomials. Novel explicit expressions for nearest-neighbor recurrence coefficients, as well as the step line case, are provided for all these polynomial families. Furthermore, new explicit expressions for type I multiple orthogonal polynomials are derived for Laguerre of the second kind and also for multiple Hermite polynomials.
\end{abstract}

\maketitle

\tableofcontents

\section{Introduction}
Multiple orthogonal polynomials constitute a class of polynomials with widespread applications across various branches of mathematics and engineering. Unlike orthogonal polynomials, which are associated with a single weight function, multiple orthogonal polynomials are linked to multiple weight functions and measures concurrently. These polynomials serve as fundamental tools in numerical analysis, approximation theory, and mathematical physics, offering robust solutions for intricate problems involving simultaneous orthogonalities.

For an in-depth understanding of the subject, one can refer to the comprehensive introduction provided in the book by Ismail \cite{Ismail}, while 
its
relation with integrable systems is elaborated upon in \cite{afm}.

Recent research has illuminated the importance of multiple orthogonal polynomials in the Favard spectral description of banded bounded semi-infinite matrices. This connection has been explored in various works such as \cite{aim, phys-scrip, Contemporary}, with further insights available in \cite{laa}. Moreover, these polynomials have been found to play a crucial role in the context of Markov chains and random walks beyond birth and death, as evidenced in \cite{CRM,finite,hypergeometric,JP}. Notably, in both of these scenarios, type I polynomials emerge as central components. Unfortunately, explicit expressions for type I multiple orthogonal polynomials remain scarce. In contrast, for type II polynomials, Rodrigues' formula yields explicit hypergeometric expressions capable of accommodating an arbitrary number of `classical' weights, as elaborated in \cite{AskeyII} and \cite{ContinuosII,Arvesu,Clasicos}, see also \cite[\S 23]{Ismail}. 

Moreover, the nearest-neighbor recurrence coefficients \eqref{NextNeighbourRecurrence} have only been fully explored in the case of $p=2$, as detailed in \cite{Clasicos}. Recent advancements, such as those in \cite{HahnI}, have provided explicit expressions for various families, including the Jacobi--Piñeiro and Laguerre polynomials of type I, specifically for $p=2$. Subsequently, in \cite{bidiagonal}, coefficients of the bidiagonal factorization for these families were determined, also focusing on $p=2$.

Expanding upon this groundwork, we extended our investigation to encompass the general case of $p\geq2$. In \cite{JPpmedidas}, we derived explicit expressions for Jacobi--Piñeiro and Laguerre polynomials of the first kind, considering $p\geq2$. Continuing in this vein, the present work focuses on providing explicit expressions for:
\begin{enumerate}[\rm i)]
	\item Laguerre polynomials of the second kind, in \S \ref{S.L2KtypeI}, and Hermite polynomials of type I, see \S \ref{S.Hermite.TypeI}.
\item Recurrence coefficients \eqref{NextNeighbourRecurrence} for Jacobi--Piñeiro, see \S \ref{S.JP}, and Laguerre polynomials of both kinds, see \S \ref{S.L1k} and \ref{S.L2Krecurrence}, along with Hermite in \S \ref{S.Hermite.recurrence}.
\end{enumerate}
Next, we give a table as a resumé of these achievements and the adequate references: 
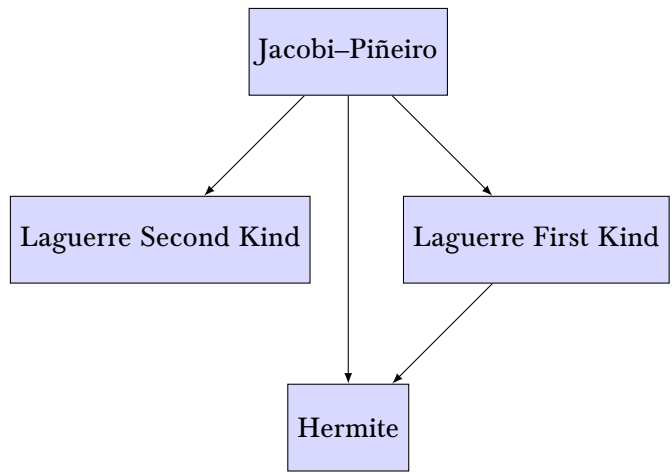
\begin{table}[h]
	\centering
	\begin{tabular}{|c|c|c|c|
		}
		\hline
		\textbf{Family} & \textbf{Type II} & \textbf{Type I} & \textbf{Recurrence} 
		\\ \hline\hline
		Jacobi--Piñeiro & {\color{green}\cmark} \cite{AskeyII} & {\color{green}\cmark} \cite{JPpmedidas}&
		{\color{green}\cmark} Here
		\\ \hline
		Laguerre First Kind& {\color{green}\cmark} \cite{AskeyII}&{\color{green}\cmark} \cite{JPpmedidas}&{\color{green}\cmark} Here
		\\ \hline
		Laguerre Second Kind & {\color{green}\cmark} \cite{AskeyII} &{\color{green}\cmark} Here &
		{\color{green}\cmark} Here
		\\ \hline
		Hermite & {\color{green}\cmark} \cite[\S 23.5]{Ismail} &{\color{green}\cmark} Here &
		{\color{green}\cmark} Here
		\\ \hline
	\end{tabular}
\end{table}

\begin{center}
	\begin{tikzpicture}[node distance=2.5cm]
		\node[fill=blue!15,block] (a) {Jacobi--Piñeiro}; 
		\node[fill=blue!15,block, right of = a, below of=a] (g) {Laguerre First Kind};
		\node[ fill=blue!15, block, left of =a, below of=a] (j) {Laguerre Second Kind};
		\node[fill=blue!15,block, below of = g, left of =g] (e) {Hermite};
		\draw[-latex] (a)--(e);
		\draw[-latex] (a)--(g);
		\draw[-latex] (g)--(e);
		\draw[-latex] (a)--(j);
	\end{tikzpicture} 
\end{center}

The following classical muliple ortohogonal polynomial families: Jacobi--Piñeiro, Laguerre of first and second kind and Hermite, are connected through limit relations as showed in the Askey scheme above, see~\cite{AskeyII}.


\subsection{Multiple Orthogonal Polynomials} (See \cite{Ismail, nikishin_sorokin}.)
Let's consider a system of $p\in\N$ weight functions $w_1,\dots,w_p:\Delta\subseteq\mathbb R\rightarrow\mathbb R^+$, a measure $\mu:\Delta\subseteq\mathbb R\rightarrow\mathbb R^+$ and a multi-index $\vec{n}=(n_1,\dots,n_p)\in\mathbb N^p_0$ with $|\vec{n}|\coloneq n_1+\dots+n_p$.

Let's examine a sequence of monic type II polynomials $B_{\vec{n}}$, where $\deg{B}\leq|\vec{n}|$, that fulfill the orthogonality relations:
\begin{align*}
 \int_{\Delta}^{}x^j B_{\vec{n}}(x)w_i(x)\d\mu(x)=0 ,
\end{align*}
for $i\in\{1,\dots,p \}$ and $j\in\{0,\dots,n_i-1\}$, and $p$ sequences of type I polynomials $A^{(1)}_{\vec{n}},\dots,A^{(p)}_{\vec{n}}$ with $\deg A^{(i)}_{\vec{n}}\leq n_i-1$ satisfying:
\begin{align}
\label{ortogonalidadTipoIContinua}
 \sum_{i=1}^p\int_{\Delta}^{}x^j A^{(i)}_{\vec{n}}(x)w_i(x)\d\mu(x)=
 \begin{cases}
 0\;\text{if}\; j=0,\dots,|\vec{n}|-2,\\
 1\;\text{if}\; j=|\vec{n}|-1.
 \end{cases}
\end{align}
These orhogonality conditions are equivalent to the biorthogonality conditions
\begin{align}
 \label{biorthogonality}
 \sum_{i=1}^p\int_{\Delta}B_{\vec{n}}(x)A^{(i)}_{\vec{m}}(x)w_i(x)\d\mu(x)=\begin{cases}
 0 \; \text{if}\;m_i\leq n_i,i\in\{1,\dots,p\},\\
 1 \;\text{if}\;|\vec{m}|=|\vec{n}|+1,\\
 0 \; \text{if}\;|\vec{n}|+1<|\vec{m}|.
 \end{cases}
\end{align}

For AT (algebraic Chebyshev) systems of weights, cf. \cite[\S 23.1.2]{Ismail,nikishin_sorokin}, the associated type II and I polynomials exist and reach the maximum degree.

\subsubsection{Near Neighbour Recurrence Relations}

Here we are partially following the notation of \cite[\S 23.1.4]{Ismail}.
Let be $\big(\pi(1),\pi(2),\dots,\pi(p)\big)$ a permutation of $(1,2,\dots,p)$, $\vec{e}_k\in\mathbb R^p$ the $k$-th vector of the canonical base in $\mathbb R^p$ and 
\begin{align*}
\vec{s}_0&\coloneq\vec{0},&\vec{s}_{j}&\begin{aligned}
	&\coloneq\sum_{i=1}^j \vec{e}_{\pi(i)},&j&\in\{1,\dots,p\}.
\end{aligned}
\end{align*}
Then, the type II and type I polynomials satisfy the following nearest-neighbor recurrence relations:
\begin{align}
\label{NextNeighbourRecurrence}
 \begin{aligned}
 	 x B_{\vec{n}}(x)&=B_{\vec{n}+\vec{e}_k}(x)+b_{\vec{n}}^0(k) B_{\vec{n}}(x)+\sum_{j=1}^p b^j_{\vec{n}}B_{\vec{n}-\vec{s}_j}(x),\\
 x A^{(i)}_{\vec{n}}(x)&=A^{(i)}_{\vec{n}-\vec{e}_k}(x)+b_{\vec{n}-\vec{e}_k}^0(k) A^{(i)}_{\vec{n}}(x)+\sum_{j=1}^p b^j_{\vec{n}+\vec{s}_{j-1}}A^{(i)}_{\vec{n}+\vec{s}_j}(x),& i&=\{1,\dots,p\}.
 \end{aligned}
\end{align}
From the biorthogonality, one gets that the recurrence coefficients can be written as:
\begin{align*}
 b_{\vec{n}}^0(k)&=\int_{\Delta} x B_{\vec{n}}(x)\left(A^{(1)}_{\vec{n}+\vec{e}_k}(x)w_1(x)+\cdots+A^{(p)}_{\vec{n}+\vec{e}_k}(x)w_p(x)\right)\d\mu(x),\\
 b_{\vec{n}}^j&=\int_{\Delta} x B_{\vec{n}}(x)\left(A^{(1)}_{\vec{n}-\vec{s}_{j-1}}(x)w_1(x)+\cdots+A^{(p)}_{\vec{n}-\vec{s}_{j-1}}(x)w_p(x)\right)\d\mu(x),& j&\in\{1,\dots,p\}.
\end{align*}

\begin{rem}
The multi-index $\vec{n}-\vec{s}_{j-1}$ corresponds 
with 
subtracting
$1$ to the $j-1$ different entries 
\begin{align*}
	 \{{n}_{\pi(1)},{n}_{\pi(2)},\dots,{n}_{\pi(j-1)}\}
\end{align*}
of the muli-index $\vec{n}$. Consequently,
\begin{align*}
 |\vec{n}+\vec{e}_k|&=|\vec{n}|+1, &
&\begin{aligned}
	 |\vec{n}-\vec{s}_{j-1}|&=|\vec{n}|-j+1,&j&=\{1,\dots,p\}.
\end{aligned}
\end{align*}
 The condition $\left(\vec{n}-\vec{s}_{j-1}\right)_i=n_i$ is equivalent to write $j\leq\pi^{-1}(i)$ while $\left(\vec{n}-\vec{s}_{j-1}\right)_i=n_i-1$ is equivalent to $j>\pi^{-1}(i)$. These conditions will be relevant later.
\end{rem}

We introduce the sets
\begin{align*}
	S(\pi,j)&\coloneq\left\{i\in\{1,\dots,p\}: j\leq \pi^{-1}(i)\right\}, &
	 S^{\textsf c}(\pi,j)&\coloneq \{1,2,\dots, p\}\setminus 	S(\pi,j).
\end{align*}
For example, for $p=4$ and the permutation
\begin{align*}
\pi=	\begin{pNiceMatrix}
		1 & 2& 3&4\\
		4&2&1 &3
	\end{pNiceMatrix}
\end{align*}
we have
\begin{align*}
	S(\pi,1)&=\{1,2,3,4\}, & 	S(\pi,2)&=\{1,2,3\}, &	S(\pi,3)&=\{1,3\}, &S(\pi,4)&=\{3\}.
\end{align*}

Type II and I polynomials can be written
\begin{align*}
 B_{\vec{n}}(x)&=\sum_{l_1=0}^{n_1}\cdots\sum_{l_p=0}^{n_p}C^{l_1,\dots,l_p}_{\vec{n}} x^{l_1+\cdots+l_p}, &
 A_{\vec{n}}^{(i)}(x)&=\sum_{l=0}^{n_i-1}C^{(i),l}_{\vec{n}} x^l,
\end{align*}
and biorthogonality conditions \eqref{biorthogonality} lead to
\begin{align}
\label{General_Coef0}
&\begin{aligned}
	 b_{\vec{n}}^0(k)&=\sum_{i=1}^p\sum_{l=0}^{n_i-1+\delta_{i,k}}C^{(i),l}_{\vec{n}+\vec{e}_k} \int_{\Delta} 
B_{\vec{n}}(x)x^{l+1}w_i(x)\d\mu(x)\\&
	=C^{(k),n_k}_{\vec{n}+\vec{e}_k}\int_{\Delta}B_{\vec{n}}(x)x^{n_k+1}w_k(x)\d\mu(x)+
	\sum_{i=1}^p C_{\vec{n}+\vec{e}_k}^{(i),n_i-1}\int_{\Delta} 
B_{\vec{n}}(x)x^{n_i}w_i(x)\d\mu(x)
\\
&=\begin{multlined}[t][.7\textwidth]C^{(k),n_k}_{\vec{n}+\vec{e}_k}\sum_{l_1=0}^{n_1}\cdots\sum_{l_p=0}^{n_p}C^{l_1,\dots,l_p}_{\vec{n}}\int_{\Delta}x^{n_k+1+l_1+\cdots+l_p}w_k(x)\d\mu(x)\\+\sum_{i=1}^p C_{\vec{n}+\vec{e}_k}^{(i),n_i-1} 
\sum_{l_1=0}^{n_1}\cdots\sum_{l_p=0}^{n_p}C^{l_1,\dots,l_p}_{\vec{n}}\int_{\Delta} x^{n_i+l_1+\cdots+l_p}w_i(x)\d\mu(x),
\end{multlined}
\end{aligned}\\
\label{General_Coefj} & \begin{aligned}
	b_{\vec{n}}^j&=\sum_{i=1}^p\sum_{l=0}^{\deg{A_{\vec{n}-\vec{s}_{j-1}}^{(i)}} }C^{(i),l}_{\vec{n}-\vec{s}_{j-1}} \int_{\Delta} B_{\vec{n}}(x)
x^{l+1}w_i(x)\d\mu(x)\\
&=\sum_{i\in S(\pi,j)} C^{(i),n_i-1}_{\vec{n}-\vec{s}_{j-1}} \int_{\Delta} B_{\vec{n}}(x)
x^{n_i}w_i(x)\d\mu(x)\\
&=\sum_{i\in S(\pi,j)} C^{(i),n_i-1}_{\vec{n}-\vec{s}_{j-1}}
\sum_{l_1=0}^{n_1}\cdots\sum_{l_p=0}^{n_p}C^{l_1,\dots,l_p}_{\vec{n}}\int_{\Delta}
x^{n_i+l_1+\cdots+l_p}w_i(x)\d\mu(x),& j&=\{1,\dots,p\},
\end{aligned}
\end{align}
since
\begin{equation*}
 \deg{A_{\vec{n}-\vec{s}_{j-1}}^{(i)}} =\begin{cases}
 n_i-2\;\text{if}\; \pi^{-1}(i)<j,\\
 n_i-1\;\text{if}\; j\leq\pi^{-1}(i).
 \end{cases}
\end{equation*}

\subsubsection{Step line Recurrence Relation}
Using the step line multi-index sequence
\begin{equation*}
 \{(0,0,\dots,0),(1,0,\dots,0),(1,1,\dots,0),\dots,(1,1,\dots,1),(2,1,\dots,1)\dots\}
\end{equation*}
we can relabel the polynomials and coefficients as follows
\begin{align*}
 B_{pm+k}&\coloneq B_{(\underbrace{m+1,\dots,m+1}_{k\,\text{times}},\underbrace{m\dots,m}_{p-k\,\text{times}})}\\
 A^{(i)}_{pm+k}&\coloneq A^{(i)}_{(\underbrace{m+1,\dots,m+1}_{k\,\text{times}},
 \underbrace{m,\dots,m}_{p-k\,\text{times}})} & i&\in\{1,\dots,p\} ,\\
 b_{pm+k}^0&\coloneq b^{0}_{(\underbrace{m+1,\dots,m+1}_{k\,\text{times}},\underbrace{m,\dots,m}_{p-k\,\text{times}})}(k+1)\\
 b_{pm+k}^j&\coloneq b^{j}_{(\underbrace{m+1,\dots,m+1}_{k\,\text{times}},\underbrace{m,\dots,m}_{p-k\,\text{times}})} & j&\in\{1,\dots,p\},
\end{align*}
with
$m\geq0$ and $k\in\{0,\dots,p-1\}$ . Then, the recurrence relations can be written as follows 
\begin{align}
\label{StepLineRecurrence}
 \begin{aligned}
 	 x B_{n}(x)&=B_{{n}+1}(x)+b_{{n}}^0 B_{{n}}(x)+\sum_{j=1}^p b^j_{{n}}B_{{n}-j}(x) , \\
 x A^{(i)}_{{n}}(x)&=A^{(i)}_{{n}-1}(x)+b_{{n}-1}^0 A^{(i)}_{{n}}(x)+\sum_{j=1}^p b^j_{{n}+{j-1}}A^{(i)}_{{n}+j}(x),& i&\in\{1,\dots,p\}.
 \end{aligned}
\end{align}
In matrix terms read
\begin{align*}
 T\begin{bNiceMatrix}
 B_0(x)\\B_1(x)\\B_2(x)\\\Vdots
 \end{bNiceMatrix}&=x\begin{bNiceMatrix}
 B_0(x)\\B_1(x)\\B_2(x)\\\Vdots
 \end{bNiceMatrix}, &
T^{\top}\begin{bNiceMatrix}
 A^{(i)}_1(x)\\[4pt]A^{(i)}_2(x)\\[4pt]A^{(i)}_3(x)\\
 \Vdots
 \end{bNiceMatrix}&=x\begin{bNiceMatrix}
 A^{(i)}_1(x)\\[4pt]A^{(i)}_2(x)\\[4pt]A^{(i)}_3(x)\\
 \Vdots
 \end{bNiceMatrix},&i\in\{1,\dots,p\},
\end{align*}
in terms of Hessenberg matrix
\begin{align}
\label{Hessenberg}
T\coloneq{\left[\begin{NiceMatrix}
		b^{0}_0 & 1 & 0&\Cdots &&&&\\
 b^{1}_1 & b^{0}_1 & 1 &\Ddots&&&&\\
&b^{1}_{2}&b^{0}_{2}&\Ddots&&&&\\ \Vdots&&\Ddots&\Ddots&&&&\\
 b^{p}_p&&&&&&&\\
 0&b^p_{p+1}&&&&&&\\
\Vdots&\Ddots[shorten-end=20pt]&\Ddots[shorten-end=20pt]&&&&&\\ &&&&&&&
	\end{NiceMatrix}\right]}.
\end{align}


\subsection{Hypergeometric Functions}
Before delving into the main results, it's important to recall that, in the context of standard orthogonality, the polynomials can be represented using generalized hypergeometric series, as discussed in \cite{andrews,slater}. These series, denoted as
\begin{align*}
 \pFq{p}{q}{a_1,\ldots,a_p}{\alpha_1,\ldots,\alpha_q}{x}:=\sum_{l=0}^{\infty}\dfrac{(a_1)_l\dots(a_p)_l}{(\alpha_1)_l\dots(\alpha_q)_l}\dfrac{x^l}{l!}.
\end{align*}
In \cite{HahnI} we found many type I polynomials families for systems of $p=2$ weight functions. The most of them were expressed through the double series known as the Kampé de Fériet functions \cite{Srivastava}
\begin{multline}
 \label{KF}
 \KF{n:r;s}{q:k;j}{(a_1,\dots,a_n):(b_1,\dots,b_r);(c_1,\dots,c_s)}{(\alpha_1,\dots,\alpha_q):(\beta_1,\dots,\beta_k);(\gamma_1,\dots,\gamma_j)}{x,y}\\
 \coloneq\sum_{l=0}^{\infty}\sum_{m=0}^{\infty}\dfrac{(a_1)_{l+m}\dots(a_n)_{l+m}}{(\alpha_1)_{l+m}\dots(\alpha_q)_{l+m}}\dfrac{(b_1)_l\dots(b_r)_l}{(\beta_1)_l\dots(\beta_k)_l}\dfrac{(c_1)_m\dots(c_s)_m}{(\gamma_1)_m\dots(\gamma_j)_m}\dfrac{x^l}{l!}\dfrac{y^m}{m!}.
\end{multline}
Here, we are going to need a generalization of these previous ones known as the multiple Kampé de Fériet functions \cite{Srivastava}
\begin{multline}
 \label{MultipleKF}
 \KF{n:r_1;\dots;r_p}{q:k_1;\dots;k_p}{(a_1,\dots,a_n):(b_1,\dots,b_{r_1});\dots;(c_1,\dots,c_{r_p})}{(\alpha_1,\dots,\alpha_q):(\beta_1,\dots,\beta_{k_1});\dots;(\gamma_1,\dots,\gamma_{k_p})}{x_1,\dots,x_p}\\
 \coloneq\sum_{l_1=0}^{\infty}\cdots\sum_{l_p=0}^{\infty}\dfrac{(a_1)_{l_1+\cdots+l_p}\cdots(a_n)_{l_1+\cdots+l_p}}{(\alpha_1)_{l_1+\cdots+l_p}\cdots(\alpha_q)_{l_1+\cdots+l_p}}\dfrac{(b_1)_{l_1}\cdots(b_{r_1})_{l_1}}{(\beta_1)_{l_1}\cdots(\beta_{k_1})_{l_1}}\cdots\dfrac{(c_1)_{l_p}\cdots(c_{k_p})_{l_p}}{(\gamma_1)_{l_p}\cdots(\gamma_{k_p})_{l_p}}\dfrac{x_1^{l_1}}{l_1!}\cdots\dfrac{x_p^{l_p}}{l_p!}.
\end{multline}
All of these functions are expresed through the Pochhammer symbols $(x)_n$, $x\in\C$ and $n\in\N_0$,
\begin{align*}
 (x)_n\coloneq\dfrac{\Gamma(x+n)}{\Gamma(x)}=\begin{cases}
 x(x+1)\cdots(x+n-1)\;\text{if}\;n\in\N,\\
 1\;\text{if}\;n=0.
 \end{cases}
\end{align*}

\section{Jacobi--Piñeiro with $p$ Weights: the Recurrence Coefficients}
\label{S.JP}

The weight functions are
\begin{align*}
&\begin{aligned}
	 w_{i}(x;\alpha_i)=&x^{\alpha_i},& i&\in\{1,\dots,p\}, 
\end{aligned}&\d\mu(x)&=(1-x)^\beta\d x, & \Delta&=[0,1],
\end{align*}
with $\alpha_1,\dots,\alpha_p,\beta>-1$ and, in order to have an AT system, $\alpha_i-\alpha_j\not\in\mathbb Z$ for $i\neq j$.
The moments are 
\begin{align}
\label{MomentJP}
\int_{0}^{1}x^{\alpha_i+k}(1-x)^\beta\d x=\dfrac{\Gamma(\beta+1)\Gamma(\alpha_i+k+1)}{\Gamma(\alpha_i+\beta+k+2)},
&& k \in \mathbb N_0.
\end{align}
The Jacobi--Piñeiro polynomials of type I are, cf. \cite{JPpmedidas},
\begin{align}
\label{JPTypeI}
 \begin{aligned}
 	 P^{(i)}_{\vec{n}}(x;\alpha_1,\dots,\alpha_p,\beta)&=
 \begin{multlined}[t ][.65\textwidth]
 (-1)^{|\vec{n}|-1}\dfrac{\prod_{q=1}^{p}(\alpha_q+\beta+|\vec{n}|)_{n_q}}{(n_i-1)!\prod_{q=1,q\neq i}^{p}(\alpha_q-\alpha_i)_{n_q}} \dfrac{\Gamma(\alpha_i+\beta+|\vec{n}|)}{\Gamma(\beta+|\vec{n}|)\Gamma(\alpha_i+1)}\\\times\pFq{p+1}{p}{-n_i+1,\alpha_i+\beta+|\vec{n}|, \{\alpha_i-\alpha_q-n_q+1\}_{q\neq i}}{\alpha_i+1,\{\alpha_i-\alpha_q+1\}_{q\neq i}}{x}
 \end{multlined}\\
 &=	\sum_{l=0}^{n_i-1}C_{\vec{n}}^{(i),l}x^l,
 \end{aligned}
\end{align}
with
\begin{align*}
C_{\vec{n}}^{(i),l}\coloneq	\frac{
\displaystyle (-1)^{|\vec{n}|-1}
\prod_{q=1}^{p}(\alpha_q+\beta+|\vec{n}|)_{n_q}}{
\displaystyle
(n_i-1)! 
\prod_{q=1,q\neq i}^{p}(\alpha_q-\alpha_i)_{n_q}} 
\frac{\Gamma(\alpha_i+\beta+|\vec{n}|)}{\Gamma(\beta+|\vec{n}|)\Gamma(\alpha_i+1)}
\frac{(-n_i+1)_l(\alpha_i+\beta+|\vec{n}|)_l}{l!(\alpha_i+1)_l}
\prod_{q=1,q\neq i}^p
\frac{(\alpha_i-\alpha_q-n_q+1)_l}{(\alpha_i-\alpha_q+1)_l}.
\end{align*}
The monic Jacobi--Piñeiro polynomials of type II are, cf. \cite[\S 3]{AskeyII},
\begin{align}
\label{JPTypeII}
 \begin{aligned}
 	 P_{\vec{n}}(x;\alpha_1,\dots,\alpha_p,\beta)&=\sum_{l_1=0}^{n_1}\cdots\sum_{l_p=0}^{n_p}
 C_{\vec{n}}^{l_1,\dots,l_p}\, x^{l_1+\cdots+l_p}\\
 C_{\vec{n}}^{l_1,\dots,l_p}&\coloneq\begin{multlined}[t][.7\textwidth]
 	(-1)^{|\vec{n}|}\prod_{q=1}^p\dfrac{(\alpha_q+1)_{n_q}}{(\alpha_q+\beta+|\vec{n}|+1)_{n_q}}\dfrac{(-n_q)_{l_q}}{l_q!}\dfrac{(\alpha_1+\beta+n_1+1)_{l_1+\cdots+l_p}}{(\alpha_1+1)_{l_1+\cdots+l_p}}
 \\
 \times\dfrac{(\alpha_1+n_1+1)_{l_2+\cdots+l_p}\cdots(\alpha_{p-1}+n_{p-1}+1)_{l_p}}{(\alpha_1+\beta+n_1+1)_{l_2+\cdots+l_p}\cdots(\alpha_{p-1}+\beta+n_1+\cdots+n_{p-1}+1)_{l_p}}
 \\ \times \dfrac{(\alpha_2+\beta+n_1+n_2+1)_{l_2+\cdots+l_p}\cdots(\alpha_{p}+\beta+|\vec{n}|+1)_{l_p}}{(\alpha_2+1)_{l_2+\cdots+l_p}\cdots(\alpha_{p}+1)_{l_p}}.
 \end{multlined}
 \end{aligned}
\end{align}

We will now endeavor to derive an explicit expression for the recurrence coefficients outlined in \eqref{NextNeighbourRecurrence}. 
In~doing so,
our first step is to establish the following summation formula:

\begin{lema}
\label{thelemma}
Let $C_{\vec{n}}^{l_1,\dots,l_p}$, as defined in \eqref{JPTypeII}, and let $m$ be a non-negative integer. Then,
\begin{align*}
\sum_{l_p=0}^{n_p}\cdots\sum_{l_1=0}^{n_1}\dfrac{(\alpha_i+n_i+m)_{l_1+\cdots+l_p}}{(\alpha_i+\beta+n_i+m+1)_{l_1+\cdots+l_p}}C^{l_1,\dots,l_p}_{\vec{n}}=(-1)^{|\vec{n}|}\dfrac{(\beta+1)_{|\vec{n}|}}{(\alpha_i+\beta+n_i+m+1)_{|\vec{n}|}}\prod_{q=1}^{p}\dfrac{(\alpha_q-\alpha_i-n_i-m+1)_{{n}_q}}{(\alpha_q+\beta+|\vec{n}|+1)_{n_q}}.
\end{align*}
\end{lema}

\begin{proof}
By substituting the previously defined type II coefficients into the multiple sum, we can express it as:
\begin{align*}
 &(-1)^{|\vec{n}|}\prod_{q=1}^p\dfrac{(\alpha_q+1)_{n_q}}{(\alpha_q+\beta+|\vec{n}|+1)_{n_q}}\sum_{l_p=0}^{n_p}\cdots\sum_{l_2=0}^{n_2}\dfrac{(-n_p)_{l_p}}{l_p!}\cdots\dfrac{(-n_2)_{l_2}}{l_2!}\dfrac{(\alpha_i+n_i+m)_{l_2+\cdots+l_p}}{(\alpha_i+\beta+n_i+m+1)_{l_2+\cdots+l_p}}\dfrac{(\alpha_1+\beta+n_1+1)_{l_2+\cdots+l_p}}{(\alpha_1+1)_{l_2+\cdots+l_p}}
 \\
 &{\times\dfrac{(\alpha_1+n_1+1)_{l_2+\cdots+l_p}\cdots(\alpha_{p-1}+n_{p-1}+1)_{l_p}}{(\alpha_1+\beta+n_1+1)_{l_2+\cdots+l_p}\cdots(\alpha_{p-1}+\beta+n_1+\cdots+n_{p-1}+1)_{l_p}}
 \dfrac{(\alpha_2+\beta+n_1+n_2+1)_{l_2+\cdots+l_p}\cdots(\alpha_{p}+\beta+|\vec{n}|+1)_{l_p}}{(\alpha_2+1)_{l_2+\cdots+l_p}\cdots(\alpha_{p}+1)_{l_p}}}\\
 &\times\underbrace{\sum_{l_1=0}^{n_1}\dfrac{(-n_1)_{l_1}}{l_1!}\dfrac{(\alpha_i+n_i+m+l_2+\cdots+l_p)_{l_1}}{(\alpha_i+\beta+n_i+m+1+l_2+\cdots+l_p)_{l_1}}\dfrac{(\alpha_1+\beta+n_1+1+l_2+\cdots+l_p)_{l_1}}{(\alpha_1+1+l_2+\cdots+l_p)_{l_1}}}_{=\pFq{3}{2}{-n_1,\alpha_i+n_i+m+l_2+\cdots+l_p,\alpha_1+\beta+n_1+1+l_2+\cdots+l_p}{\alpha_i+\beta+n_i+m+1+l_2+\cdots+l_p,\alpha_1+1+l_2+\cdots+l_p}{1}}.
\end{align*}
Observing the summation labeled by $l_1$, it becomes evident that it corresponds to a $_3F_2$ hypergeometric function, which adheres to the Pfaff--Saalschütz formula
\begin{align*}
 \pFq{3}{2}{-n,a,b,}{c,-n+a+b+1-c}{1}=\dfrac{(c-a)_n(c-b)_n}{(c)_n(c-a-b)_n}.
\end{align*}
Thus, we can utilize this formula to determine
\begin{multline*}
 \pFq{3}{2}{-n_1,\alpha_i+n_i+m+l_2+\cdots+l_p,\alpha_1+\beta+n_1+1+l_2+\cdots+l_p}{\alpha_i+\beta+n_i+m+1+l_2+\cdots+l_p,\alpha_1+1+l_2+\cdots+l_p}{1}\\
 ={\dfrac{(\beta+1)_{n_1}(\alpha_i-\alpha_1+n_i-n_1+m)_{n_1}}{(\alpha_i+\beta+n_i+m+1+l_2+\cdots+l_p)_{n_1}(-\alpha_1-n_1-l_2-\cdots-l_p)_{n_1}}}.
\end{multline*}
Substituting the previous expression and clearing, we obtain:
\begin{multline*}
 (-1)^{|\vec{n}|}\prod_{q=1}^p\dfrac{(\alpha_q+1)_{n_q}}{(\alpha_q+\beta+|\vec{n}|+1)_{n_q}}{\dfrac{(\beta+1)_{n_1}(\alpha_1-\alpha_i-n_i-m+1)_{n_1}}{(\alpha_i+\beta+n_i+m+1)_{n_1}(\alpha_1+1)_{n_1}}}\sum_{l_p=0}^{n_p}\cdots\sum_{l_3=0}^{n_3}\dfrac{(-n_p)_{l_p}}{l_p!}\cdots\dfrac{(-n_3)_{l_3}}{l_2!} 
  \\
 \times\dfrac{(\alpha_i+n_i+m)_{l_3+\cdots+l_p}}{(\alpha_i+\beta+n_i+n_1+m+1)_{l_3+\cdots+l_p}}\dfrac{(\alpha_2+\beta+n_1+n_2+1)_{l_3+\cdots+l_p}}{(\alpha_2+1)_{l_3+\cdots+l_p}}
 \\
 \times\dfrac{(\alpha_2+n_2+1)_{l_3+\cdots+l_p}\cdots(\alpha_{p-1}+n_{p-1}+1)_{l_p}}{(\alpha_2+\beta+n_1+n_2+1)_{l_3+\cdots+l_p}\cdots(\alpha_{p-1}+\beta+n_1+\cdots+n_{p-1}+1)_{l_p}}
  \\
\times
 \dfrac{(\alpha_3+\beta+n_1+n_2+n_3+1)_{l_3+\cdots+l_p}\cdots(\alpha_{p}+\beta+|\vec{n}|+1)_{l_p}}{(\alpha_3+1)_{l_3+\cdots+l_p}\cdots(\alpha_{p}+1)_{l_p}}\\
 \times\underbrace{\sum_{l_2=0}^{n_2}\dfrac{(-n_2)_{l_2}}{l_2!}\dfrac{(\alpha_i+n_i+m+l_3+\cdots+l_p)_{l_2}}{(\alpha_i+\beta+n_i+n_1+m+1+l_3+\cdots+l_p)_{l_2}}\dfrac{(\alpha_2+\beta+n_1+n_2+1+l_3+\cdots+l_p)_{l_2}}{(\alpha_2+1+l_3+\cdots+l_p)_{l_2}}}_{=\pFq{3}{2}{-n_2,\alpha_i+n_i+m+l_3+\cdots+l_p,\alpha_2+\beta+n_1+n_2+1+l_3+\cdots+l_p}{\alpha_i+\beta+n_i+n_1+m+1+l_3+\cdots+l_p,\alpha_2+1+l_3+\cdots+l_p}{1}}.
\end{multline*}
Now, through a parameter change, we transform the current $(p-1)$ multiple sum into an equivalent expression to the previous $p$ multiple sum:
\begin{align*}
 (n_1,\dots,n_p)\in\mathbb N_0^p&\rightarrow(n_2,\dots,n_p)\in\mathbb N_0^{p-1},\\
 (\alpha_1,\dots,\alpha_p)\in\mathbb R^{p}&\rightarrow(\alpha_2,\dots,\alpha_p)\in\mathbb R^{p-1},\\
 \alpha_i&\rightarrow\alpha_i,\\
 \beta&\rightarrow\beta+n_1.
\end{align*}
Hence, by applying this procedure recursively, we can iteratively reduce all the sums, yielding:
\begin{multline*}
 (-1)^{|\vec{n}|} \prod_{q=1}^p\dfrac{(\alpha_q+1)_{n_q}}{(\alpha_q+\beta+|\vec{n}|+1)_{n_q}}
 {\dfrac{(\beta+1)_{n_1}(\alpha_1-\alpha_i-n_i-m+1)_{n_1}}{(\alpha_i+\beta+n_i+m+1)_{n_1}(\alpha_1+1)_{n_1}}}\\
 \times{\dfrac{(\beta+n_1+1)_{n_2}(\alpha_2-\alpha_i-n_i-m+1)_{n_2}}{(\alpha_i+\beta+n_i+n_1+m+1)_{n_2}(\alpha_2+1)_{n_2}}}\cdots{\dfrac{(\beta+n_1+\cdots+n_{p-1}+1)_{n_p}(\alpha_p-\alpha_i-n_p-m+1)_{n_p}}{(\alpha_i+\beta+n_i+n_1+\cdots+n_{p-1}+m+1)_{n_p}(\alpha_p+1)_{n_p}}}\\
 =(-1)^{|\vec{n}|}
 {\dfrac{(\beta+1)_{|\vec{n}|}}{(\alpha_i+\beta+n_i+m+1)_{|\vec{n}|}}}\prod_{q=1}^p\dfrac{(\alpha_q-\alpha_i-n_i-m+1)_{n_q}}{(\alpha_q+\beta+|\vec{n}|+1)_{n_q}}.
\end{multline*}
\end{proof}

Equipped with this lemma, we are now prepared to establish that:
\begin{teo}[Explicit expressions for $p$ weights Jacobi--Piñeiro's recurrence coefficients]
The Jacobi–Piñeiro multiple orthogonal polynomials of type I, as expressed in \eqref{JPTypeI}, and type II, as defined in \eqref{JPTypeII}, each adhere to their respective nearest-neighbor recurrence relations, as outlined in \eqref{NextNeighbourRecurrence}, with respect to the coefficients:
\begin{align}
\label{JPRecurrence}
\begin{aligned}
	b_{\vec{n}}^0(k)
&=\begin{multlined}[t][.8\textwidth]
\dfrac{(\alpha_k+n_k+1)}{(\alpha_k+\beta+n_k+|\vec{n}|+2)}\prod_{q=1}^p\dfrac{(\alpha_k-\alpha_q+n_k+1)}{(\alpha_k-\alpha_q+n_k+1-n_q)}\\+\sum_{i=1}^p \dfrac{(\alpha_i+n_i)(\alpha_k+\beta+n_k+|\vec{n}|+1)}{(\alpha_i+\beta+n_i+|\vec{n}|)_2(\alpha_i-\alpha_k-n_k+n_i-1)}\dfrac{\prod_{q=1}^p(\alpha_i-\alpha_q+n_i)}{\prod_{q=1,q\neq i}^p(\alpha_i-\alpha_q+n_i-n_q)},
\end{multlined}\\
 b_{\vec{n}}^j
&=\begin{multlined}[t][.8\textwidth]
	(\beta+|\vec{n}|-j+1)_j\prod_{q=1}^p\dfrac{(\alpha_q+\beta+|\vec{n}|-j+1)_{j}}{(\alpha_q+\beta+|\vec{n}|-j+n_q+1)_{j}}\sum_{i\in S(\pi,j)}\dfrac{(\alpha_i+n_i)}{(\alpha_i+\beta+n_i+|\vec{n}|-j)_{j+2}}\\
\times\dfrac{\prod_{q=1}^p(\alpha_i-\alpha_q+n_i)}{\prod_{q\in S(\pi,j),q\neq i}(\alpha_i-\alpha_q+n_i-n_q)\prod_{q\in S^{\textsf c}(\pi,j)}(\alpha_q+\beta+|\vec{n}|-j+n_q)} ,\quad j\in\{1,\dots,p\}.
\end{multlined}
\end{aligned}
\end{align}
\end{teo}

\begin{proof}
Substituting the Jacobi–Piñeiro moments from \eqref{MomentJP} into the expressions \eqref{General_Coef0} and \eqref{General_Coefj}, we obtain:
\begin{align*}
 b_{\vec{n}}^0(k)&=\begin{multlined}[t][.9\textwidth]
 	C^{(k),n_k}_{\vec{n}+\vec{e}_k}\sum_{l_1=0}^{n_1}\cdots\sum_{l_p=0}^{n_p}C^{l_1,\dots,l_p}_{\vec{n}}\dfrac{\Gamma(\beta+1)\Gamma(\alpha_k+n_k+l_1+\cdots+l_p+2)}{\Gamma(\alpha_k+\beta+n_k+l_1+\cdots+l_p+3)}\\+\sum_{i=1}^p C_{\vec{n}+\vec{e}_k}^{(i),n_i-1} 
\sum_{l_1=0}^{n_1}\cdots\sum_{l_p=0}^{n_p}C^{l_1,\dots,l_p}_{\vec{n}}\dfrac{\Gamma(\beta+1)\Gamma(\alpha_i+n_i+l_1+\cdots+l_p+1)}{\Gamma(\alpha_i+\beta+n_i+l_1+\cdots+l_p+2)}
 \end{multlined}\\
&=\begin{multlined}[t][.9\textwidth]
	\dfrac{\Gamma(\beta+1)\Gamma(\alpha_k+n_k+2)}{\Gamma(\alpha_k+\beta+n_k+3)}C^{(k),n_k}_{\vec{n}+\vec{e}_k}\sum_{l_1=0}^{n_1}\cdots\sum_{l_p=0}^{n_p}\dfrac{(\alpha_k+n_k+2)_{l_1+\cdots+l_p}}{(\alpha_k+\beta+n_k+3)_{l_1+\cdots+l_p}}C^{l_1,\dots,l_p}_{\vec{n}}\\+\sum_{i=1}^p \dfrac{\Gamma(\beta+1)\Gamma(\alpha_i+n_i+1)}{\Gamma(\alpha_i+\beta+n_i+2)}C_{\vec{n}+\vec{e}_k}^{(i),n_i-1} 
\sum_{l_1=0}^{n_1}\cdots\sum_{l_p=0}^{n_p}\dfrac{(\alpha_i+n_i+1)_{l_1+\cdots+l_p}}{(\alpha_i+\beta+n_i+2)_{l_1+\cdots+l_p}}C^{l_1,\dots,l_p}_{\vec{n}},
\end{multlined}\\
 b_{\vec{n}}^j&=\sum_{i\in S(\pi,j)} C^{(i),n_i-1}_{\vec{n}-\vec{s}_{j-1}}
\sum_{l_1=0}^{n_1}\cdots\sum_{l_p=0}^{n_p}C^{l_1,\dots,l_p}_{\vec{n}}\dfrac{\Gamma(\beta+1)\Gamma(\alpha_i+n_i+l_1+\cdots+l_p+1)}{\Gamma(\alpha_i+\beta+n_i+l_1+\cdots+l_p+2)}\\
&=\sum_{i\in S(\pi,j)} \dfrac{\Gamma(\beta+1)\Gamma(\alpha_i+n_i+1)}{\Gamma(\alpha_i+\beta+n_i+2)}C^{(i),n_i-1}_{\vec{n}-\vec{s}_{j-1}}
\sum_{l_1=0}^{n_1}\cdots\sum_{l_p=0}^{n_p}\dfrac{(\alpha_i+n_i+1)_{l_1+\cdots+l_p}}{(\alpha_i+\beta+n_i+2)_{l_1+\cdots+l_p}}C^{l_1,\dots,l_p}_{\vec{n}},\quad j\in\{1,\dots,p\}.
\end{align*}
Now, we can utilize Lemma \ref{thelemma} to simplify the sums labeled by $l_1,\dots,l_p$ as:
\begin{align*}
\sum_{l_p=0}^{n_p}\cdots\sum_{l_1=0}^{n_1}\dfrac{(\alpha_i+n_i+m)_{l_1+\cdots+l_p}}{(\alpha_i+\beta+n_i+m+1)_{l_1+\cdots+l_p}}C^{l_1,\dots,l_p}_{\vec{n}}=(-1)^{|\vec{n}|}\dfrac{(\beta+1)_{|\vec{n}|}}{(\alpha_i+\beta+n_i+m+1)_{|\vec{n}|}}\prod_{q=1}^{p}\dfrac{(\alpha_q-\alpha_i-n_i-m+1)_{{n}_q}}{(\alpha_q+\beta+|\vec{n}|+1)_{n_q}}
\end{align*}
and get
\begin{align*}
 b_{\vec{n}}^0(k)&\begin{multlined}[t][.9\textwidth]
 	=(-1)^{|\vec{n}|}\dfrac{(\beta+1)_{|\vec{n}|}}{(\alpha_k+\beta+n_k+3)_{|\vec{n}|}}\prod_{q=1}^{p}\dfrac{(\alpha_q-\alpha_k-n_k-1)_{{n}_q}}{(\alpha_q+\beta+|\vec{n}|+1)_{n_q}}\dfrac{\Gamma(\beta+1)\Gamma(\alpha_k+n_k+2)}{\Gamma(\alpha_k+\beta+n_k+3)}C^{(k),n_k}_{\vec{n}+\vec{e}_k}\\
 +\sum_{i=1}^p (-1)^{|\vec{n}|}\dfrac{(\beta+1)_{|\vec{n}|}}{(\alpha_i+\beta+n_i+2)_{|\vec{n}|}}\prod_{q=1}^{p}\dfrac{(\alpha_q-\alpha_i-n_i)_{{n}_q}}{(\alpha_q+\beta+|\vec{n}|+1)_{n_q}}\dfrac{\Gamma(\beta+1)\Gamma(\alpha_i+n_i+1)}{\Gamma(\alpha_i+\beta+n_i+2)}C_{\vec{n}+\vec{e}_k}^{(i),n_i-1} ,
 \end{multlined}\\
 b_{\vec{n}}^j&=\sum_{i\in S(\pi,j)} (-1)^{|\vec{n}|}\dfrac{(\beta+1)_{|\vec{n}|}}{(\alpha_i+\beta+n_i+2)_{|\vec{n}|}}\prod_{q=1}^{p}\dfrac{(\alpha_q-\alpha_i-n_i)_{{n}_q}}{(\alpha_q+\beta+|\vec{n}|+1)_{n_q}}\dfrac{\Gamma(\beta+1)\Gamma(\alpha_i+n_i+1)}{\Gamma(\alpha_i+\beta+n_i+2)}C^{(i),n_i-1}_{\vec{n}-\vec{s}_{j-1}},
\end{align*}
for $j\in\{1,\dots,p\}$.
Finally, by substituting the type I coefficients from \eqref{JPTypeI} and simplifying, we arrive at the expressions given in \eqref{JPRecurrence}.
\end{proof}

Let us now introduce the following notation:
\begin{equation}\label{eq:alphas}
 \underbrace{\alpha_0}_{\coloneq-1},\alpha_1,\alpha_2,\dots,\alpha_p,\;\underbrace{\alpha_{p+1}}_{\coloneq\alpha_{1}+1},\underbrace{\alpha_{p+2}}_{\coloneq\alpha_{2}+1},\dots,\underbrace{\alpha_{2p}}_{\coloneq\alpha_{p}+1},\;\underbrace{\alpha_{2p+1}}_{\coloneq\alpha_{1}+2},\underbrace{\alpha_{2p+2}}_{\coloneq\alpha_{2}+2},\dots
\end{equation}

This enables us to express the step-line recurrence coefficients from \eqref{StepLineRecurrence} in a unified formula as follows:

\begin{cor}[Step line Jacobi--Piñeiro's recurrence coefficients]
The Jacobi–Piñeiro type I, as given in \eqref{JPTypeI}, and type II, as defined in \eqref{JPTypeII}, multiple orthogonal polynomials in the stepline, adhere to a recurrence relation of the form \eqref{StepLineRecurrence} with respect to the coefficients.
\begin{multline}
\label{JPStepLineRecurrence}
 b_{{pm+k}}^j=\dfrac{(\alpha_{k+1}+\beta+(p+1)m+k+1-j)}{\prod_{q=p+k+2-j}^{p+k+1}(\alpha_q+\beta+(p+1)m+k-j)}\dfrac{(\beta+pm+k+1-j)_j\prod_{q=1}^p(\alpha_q+\beta+pm+k+1-j)_{j}}{\prod_{q=k+1}^{p+k}(\alpha_q+\beta+(p+1)m+k+1-j)_{j}}\\
 \times\sum_{i=k+1}^{p+k+1-j}\dfrac{(\alpha_i+m)}{(\alpha_i+\beta+(p+1)m+k-j)_{j+2}}\dfrac{\prod_{q=1}^p(\alpha_i-\alpha_q+m)}{\prod_{q=k+1,q\neq i}^{p+k+1-j}(\alpha_i-\alpha_q)},
\end{multline}
for $j\in\{0,1,\dots,p\}$, $m\geq0$ and $k\in\{0,\dots,p-1\}$.
\end{cor}

\section{Laguerre of First Kind with $p$ Weights: the Recurrence Coefficients}
\label{S.L1k}
In this case we have
\begin{align*}
 &\begin{aligned}
 	 w_{i}(x;\alpha_i)&=\Exp{-x}x^{\alpha_i}, & i&\in\{1,\dots,p\},
 \end{aligned}&\d\mu(x)&=\d x, & \Delta&=[0,\infty),
\end{align*}
with $\alpha_1,\dots,\alpha_p>-1$ and, in order to have an AT system, $\alpha_i-\alpha_j\not\in\mathbb Z$ for $i\neq j$.

The Laguerre of first kind multiple orthogonal polynomials can be obtained as a limit of the Jacobi--Piñeiro polynomials. The type I are, cf. \cite{JPpmedidas},
\begin{align}
\label{LaguerreFKTypeI}
 \begin{aligned}
 & L^{(i)}_{\vec{n}}(x;\alpha_1,\dots,\alpha_p) =\lim_{\beta\rightarrow\infty}\dfrac{\Gamma(\beta+|\vec{n}|)}{\prod_{q=1}^{p}(\alpha_q+\beta+|\vec{n}|)_{n_q}\Gamma(\alpha_i+\beta+|\vec{n}|)}P_{\vec{n}}^{(i)}\left(\dfrac{x}{\beta};\alpha_1,\dots,\alpha_p,\beta\right)\\
 &\hspace{1.25cm} =\lim_{\beta\rightarrow\infty}\dfrac{1}{\beta^{\alpha_i+|\vec{n}|}}P_{\vec{n}}^{(i)}\left(\dfrac{x}{\beta};\alpha_1,\dots,\alpha_p,\beta\right)\\
 &\hspace{1.25cm} =\dfrac{(-1)^{|\vec{n}|-1}}{(n_i-1)!\prod_{q=1,q\neq i}^{p}(\alpha_q-\alpha_i)_{n_q}\Gamma(\alpha_i+1)} \pFq{p}{p}{-n_i+1,\{ \alpha_i-\alpha_q-n_q+1\}_{q\neq i}}{\alpha_i+1,\{\alpha_i-\alpha_q+1\}_{q\neq i}}{x}\\
 &\hspace{1.25cm} =\dfrac{(-1)^{|\vec{n}|-1}}{(n_i-1)!\prod_{q=1,q\neq i}^{p}(\alpha_q-\alpha_i)_{n_q}\Gamma(\alpha_i+1)}\sum_{l=0}^{n_i-1}{\dfrac{(-n_i+1)_l}{l!}\dfrac{1}{(\alpha_i+1)_l}\prod_{q=1,q\neq i}^p\dfrac{(\alpha_i-\alpha_q-n_q+1)_l}{(\alpha_i-\alpha_q+1)_l}}x^l.
 \end{aligned}
\end{align}
The monic type II polynomials are, cf. \cite[\S 4]{AskeyII},
\begin{align}
\label{LaguerreFKTypeII}
 \begin{aligned}
 	& L_{\vec{n}}(x;\alpha_1,\dots,\alpha_p)=\lim_{\beta\rightarrow\infty}\beta^{|\vec{n}|}P_{\vec{n}}\left(\dfrac{x}{\beta},\alpha_1,\dots,\alpha_p,\beta\right)=\sum_{l_1=0}^{n_1}\cdots\sum_{l_p=0}^{n_p}
 C_{\vec{n}}^{l_1,\dots,l_p}\, x^{l_1+\cdots+l_p}\\
 & C_{\vec{n}}^{l_1,\dots,l_p} \coloneq(-1)^{|\vec{n}|}\prod_{q=1}^p{(\alpha_q+1)_{n_q}}\dfrac{(-n_q)_{l_q}}{l_q!}\dfrac{1}{(\alpha_1+1)_{l_1+\cdots+l_p}}
{\dfrac{(\alpha_1+n_1+1)_{l_2+\cdots+l_p}\cdots(\alpha_{p-1}+n_{p-1}+1)_{l_p}}{(\alpha_2+1)_{l_2+\cdots+l_p}\cdots(\alpha_{p}+1)_{l_p}}}.
 \end{aligned}
\end{align}


\begin{teo}[Explicit expressions for $p$ weights Laguerre of first kind recurrence coefficients]
The Laguerre of the first kind multiple orthogonal polynomials of type I, as described in \eqref{LaguerreFKTypeI}, and type II, as defined in \eqref{LaguerreFKTypeII}, each adhere to their respective nearest-neighbor recurrence relations, as outlined in \eqref{NextNeighbourRecurrence}, with respect to the coefficients:
\begin{align}
 \label{LaguerreFKRecurrence}
\begin{aligned}
	 b_{\vec{n}}^0(k)
=&{(\alpha_k+n_k+1)}\prod_{q=1}^p\dfrac{(\alpha_k-\alpha_q+n_k+1)}{(\alpha_k-\alpha_q+n_k+1-n_q)}
 \\ 
 & \hspace{4cm} +\sum_{i=1}^p \dfrac{(\alpha_i+n_i)}{(\alpha_i-\alpha_k-n_k+n_i-1)}\dfrac{\prod_{q=1}^p(\alpha_i-\alpha_q+n_i)}{\prod_{q=1,q\neq i}^p(\alpha_i-\alpha_q+n_i-n_q)},\\
 b_{\vec{n}}^j
=&\sum_{i\in S(\pi,j)}\dfrac{(\alpha_i+n_i)\prod_{q=1}^p{(\alpha_i-\alpha_q+n_i)}}{\prod_{
		q\in S(\pi,j),q\neq i}(\alpha_i-\alpha_q+n_i-n_q)} ,\quad j\in\{1,\dots,p\}.
\end{aligned}
\end{align}
\end{teo}

\begin{proof}
The corresponding limits described in \eqref{LaguerreFKTypeI} and \eqref{LaguerreFKTypeII}, which establish the connection between the Jacobi--Piñeiro and Laguerre of first kind multiple orthogonal polynomials, imply the recurrence coefficients for the Laguerre polynomials can be obtained from the Jacobi--Piñeiro ones \eqref{JPRecurrence} just by applying the limit
\begin{align*}
 \begin{aligned}
 	& \lim_{\beta\rightarrow\infty}\beta^{j+1}b_{\vec{n}}^j,&j&\in\{0,1,\dots,p\}.
 \end{aligned}
\end{align*}
This application is immediate and yields the aforementioned expression.
\end{proof}

Now let's mantain the previous notation in \eqref{eq:alphas}
to write the step line coefficients as:
\begin{cor}[Step line Laguerre of the first kind recurrence coefficients]
 The Laguerre of first kind multiple orthogonal polynomials of type I, as in \eqref{LaguerreFKTypeI}, and of type II, as in \eqref{LaguerreFKTypeII}, in the step line satisfy a recurrence relation of the form \eqref{StepLineRecurrence} respect to the coefficients
\begin{align}
\label{LaguerreFKStepLineRecurrence} b_{{pm+k}}^j=&\sum_{i=k+1}^{p+k+1-j}\dfrac{(\alpha_i+m)\prod_{q=1}^p(\alpha_i-\alpha_q+m)}{\prod_{q=k+1,q\neq i}^{p+k+1-j}(\alpha_i-\alpha_q)} ,
\end{align}
for $j\in\{0,1,\dots,p\}$, $m\geq0$ and $k\in\{0,\dots,p-1\}$.
\end{cor}


\section{Laguerre of Second Kind with $p$ Weights: Type I Polynomials and the Recurrence Coefficients}

The weight functions for this family are
\begin{align}
 \label{WeightsLaguerreSK}
 &\begin{aligned}
 	w_i(x;c_i,\alpha_0)&\coloneq x^{\alpha_0}\Exp{-c_ix}, & i&\in\{1,\dots,p\}, 
 \end{aligned}& \d\mu(x)&=\d x,& \Delta&=[0,\infty),
\end{align}
with $\alpha_0>-1, c_1,\ldots,c_p>0$ and, in order to have an AT system, $c_i\neq c_j$ for $i\neq j$.
The corresponding monic type II multiple orthogonal polynomials are, cf. \cite[\S 4]{AskeyII},
\begin{align}
 \label{LaguerreSKTypeII}
\begin{aligned}
	 L_{\vec{n}}(x;c_1,\ldots,c_p,\alpha_0)&=\lim_{t\rightarrow\infty}(-t)^{|\vec{n}|}P_{\vec{n}}\left(1-\dfrac{x}{t};c_1t,\dots,c_pt,\alpha_0\right)\\
 &=(-1)^{|\vec{n}|}(\alpha_0+1)_{|\vec{n}|}\sum_{l_1=0}^{n_1}\cdots\sum_{l_p=0}^{n_p}\dfrac{1}{(\alpha_0+1)_{l_1+\cdots+l_p}}\prod_{q=1}^p\dfrac{(-n_q)_{l_q}c_q^{l_q-n_q}}{l_q!}x^{l_1+\cdots+l_p}.
\end{aligned}
\end{align}

\subsection{Explicit Hypergeometric Expressions for the Type I polynomials}
\label{S.L2KtypeI}

In \cite[\S 8]{HahnI}, we established for a system with $p=2$ weight functions that they can be represented by the following hypergeometric expression in the form of a Kampé de Fériet series, as shown in \eqref{KF}:
\begin{multline*}
 L_{(n_1,n_2)}^{(i)}(x;c_1,c_2,\alpha_0)=\dfrac{(-1)^{{n}_i-1}(n_1+n_2-2)!}{(n_1-1)!(n_2-1)!\Gamma(\alpha_0+n_1+n_2)}c_i^{\alpha_0+n_1+n_2}\left(\dfrac{\hat{c}_i}{\hat{c}_i-c_i}\right)^{n_1+n_2-1}\\ \times\KF{2:0;0}{1:1;0}{-n_i+1,\,\alpha_0+1:--;--}{-n_1-n_2+2:\alpha_0+1;--}{\dfrac{(c_i-\hat{c}_i)c_i}{\hat{c}_i}x,-\dfrac{(c_i-\hat{c}_i)}{\hat{c}_i}},
\end{multline*}
for $i\in\{1,2\}$. Here we have defined $\hat{c}_i\coloneq\delta_{i,1}c_2+\delta_{i,2}c_1$.

An extension of this result to an arbitrary number of weights is provided by a multiple Kampé de Fériet series, as depicted in \eqref{MultipleKF}.
\begin{teo}[Explicit hypergeometric expressions for the type Laguerre of the second kind]
\label{LaguerreSKTypeITheorem}
The Laguerre of second kind multiple orthogonal polynomials of type I are
\begin{align}
\label{LaguerreSKTypeI}
\begin{aligned}
	&
L_{\vec{n}}^{(i)}(x;c_1,\dots,c_p,\alpha_0) =
\begin{multlined}[t][.67\textwidth]
	\dfrac{(-1)^{n_i-1}c_i^{\alpha_0+|\vec{n}|}(\alpha_0+|\vec{n}|-n_i+1)_{n_i-1}}{(n_i-1)!\Gamma(\alpha_0+|\vec{n}|)}\prod_{q=1,q\neq i}^{p}\left(\dfrac{c_q}{c_q-c_i}\right)^{n_q}\\
\times\KF{1:0;1;\cdots;1}{1:0;0;\cdots;0}{-n_i+1:--;\{n_q\}_{q\neq i}}{\alpha_0+|\vec{n}|-n_i+1:--;\cdots;--}{{c_i x},\left\{\dfrac{c_i}{c_i-c_q}\right\}_{q\neq i}}
\end{multlined}\\
& 
\hspace{1cm}
=\begin{multlined}[t][.6\textwidth]
	\dfrac{(-1)^{n_i-1}c_i^{\alpha_0+|\vec{n}|}(\alpha_0+|\vec{n}|-n_i+1)_{n_i-1}}{(n_i-1)!\Gamma(\alpha_0+|\vec{n}|)}\prod_{q=1,q\neq i}^{p}\left(\dfrac{c_q}{c_q-c_i}\right)^{n_q}
	\\ \times
\sum_{l_1=0}^{n_i-1}\sum_{l_2=0}^{n_i-1-l_1}\cdots\sum_{l_p=0}^{n_i-1-l_1-\cdots-l_{p-1}}\dfrac{(-n_i+1)_{l_1+\cdots+l_p}}{(\alpha_0+|\vec{n}|-n_i+1)_{l_1+\cdots+l_p}}\dfrac{c_i^{l_1+\cdots+l_p}}{l_{i}!}
 \prod_{q=1,q\neq i}^p\dfrac{(n_q)_{l_q}}{l_q!(c_i-c_q)^{l_q}}\,x^{l_i},
\end{multlined}
\end{aligned}
\end{align}
for $i\in\{1,\dots,p\}$, and can be obtained from the Jacobi--Piñeiro polynomials \eqref{JPTypeI} through the limit
\begin{align}
\label{LaguerreSKasJPLimitTypeI}
\begin{aligned}
	 L_{\vec{n}}^{(i)}(x;c_1,\dots,c_p,\alpha_0)&=\lim_{t\rightarrow\infty}\dfrac{(-1)^{|\vec{n}|-1}}{t^{\alpha_0+|\vec{n}|}}P_{\vec{n}}^{(i)}\left(1-\dfrac{x}{t}; c_1t,\dots,c_pt,\alpha_0\right),&i&\in\{1,\dots,p\}.
\end{aligned}
\end{align}
\end{teo}

\begin{proof}

We will divide the proof into two parts. Firstly, we aim to demonstrate that the polynomials defined by the previous limit indeed correspond to the Laguerre of the second kind, type I polynomials. This can be achieved by verifying that they satisfy the orthogonality conditions given in \eqref{ortogonalidadTipoIContinua} with respect to the weight functions outlined in \eqref{WeightsLaguerreSK}.

The Jacobi--Piñeiro type I polynomials satisfy the orthogonality conditions
\begin{align*}
 \sum_{i=1}^p\int_{0}^1 (x-1)^k P_{\vec{n}}^{(i)}(x;\alpha_1,\dots,\alpha_p,\beta) x^{\alpha_i}(1-x)^{\beta}\d x=\begin{cases}
 0,\;\text{if}\; k\in\{0,\dots,|\vec{n}|-2\},\\
 1,\;\text{if}\; k=|\vec{n}|-1.
 \end{cases}
\end{align*}
Under the change of variables
\begin{align*}
\begin{aligned}
	 x&\rightarrow 1-\dfrac{x}{t},& \alpha_i&\rightarrow c_it, & \beta&\rightarrow\alpha_0 ,
\end{aligned}
\end{align*}
the previous expression becomes
\begin{align*}
 \sum_{i=1}^p
 \int_{0}^t x^k \dfrac{(-1)^{|\vec{n}|-1}}{t^{\alpha_0+|\vec{n}|}}P_{\vec{n}}^{(i)}\left(1-\dfrac{x}{t};c_1 t,\dots,c_p t,\alpha_0\right) x^{\alpha_0}\left(1-\dfrac{x}{t}\right)^{c_it}\d x=\begin{cases}
 0,\;\text{if}\;k\in\{0,\dots,|\vec{n}|-2\},\\
 1,\;\text{if}\; k=|\vec{n}|-1.
 \end{cases}
\end{align*}

Now, our objective is to apply the limit as $t\rightarrow\infty$ to both sides of the preceding equation. To ensure that the limit can be interchanged with the integral, we need to establish that the modulus of the expression within the integral is bounded by an integrable function for all $t>0$.

On one hand, we have that
\begin{align*}
\begin{aligned}
	 \left(1-\dfrac{x}{t}\right)^{c_it} I_{[0,t)}&\leq \mathrm e^{-c_i x},&t&>0, & x&\in\mathbb R^+.
\end{aligned}
\end{align*}
On the other hand, the polynomial
\begin{align*}
\dfrac{(-1)^{|\vec{n}|-1}}{t^{\alpha_0+|\vec{n}|}}P_{\vec{n}}^{(i)}\left(1-\dfrac{x}{t};c_1 t,\dots,c_p t,\alpha_0\right)
\end{align*}
is bounded $ t>0$ since the limit $t\to\infty$ exists as we will show.
So the expression is bounded and, by Lebesgue's dominated convergence, we can interchange the limit to get that
\begin{align*}
 \sum_{i=1}^p
 \int_{0}^\infty x^k
 \underbrace{\lim_{t\rightarrow\infty} \dfrac{(-1)^{|\vec{n}|-1}}{t^{\alpha_0+|\vec{n}|}}P_{\vec{n}}^{(i)}\left(1-\dfrac{x}{t};c_1 t,\dots,c_p t,\alpha_0\right)}_{=L^{(i)}_{\vec{n}}(x;c_1,\ldots,c_p,\alpha_0)\;\text{by definition}} \underbrace{\lim_{t\rightarrow\infty}x^{\alpha_0}\left(1-\dfrac{x}{t}\right)^{c_it}}_{=x^{\alpha_0} \exp(-c_i x)}\d x=\begin{cases}
 0,\;\text{if}\; k\in\{0,\dots,|\vec{n}|-2\},\\
 1,\;\text{if}\; k=|\vec{n}|-1.
 \end{cases}
\end{align*}

Once we have proven this, the second part of the proof will consist on calculate such limit \eqref{LaguerreSKasJPLimitTypeI} to arrive to mentioned expression \eqref{LaguerreSKTypeI}.

Remember the Jacobi--Piñeiro type I polynomials are
\begin{multline*}
 P^{(i)}_{\vec{n}}(x;\alpha_1,\dots,\alpha_p,\beta) =(-1)^{|\vec{n}|-1}\dfrac{\prod_{q=1}^{p}(\alpha_q+\beta+|\vec{n}|)_{n_q}}{(n_i-1)!\prod_{q=1,q\neq i}^{p}(\alpha_q-\alpha_i)_{n_q}} \dfrac{\Gamma(\alpha_i+\beta+|\vec{n}|)}{\Gamma(\beta+|\vec{n}|)\Gamma(\alpha_i+1)}\\\times \sum_{l=0}^{n_i-1}{\dfrac{(-n_i+1)_l(\alpha_i+\beta+|\vec{n}|)_l \prod_{q=1,q\neq i}^p(\alpha_i-\alpha_q-n_q+1)_l}{l!(\alpha_i+1)_l\prod_{q=1,q\neq i}^p(\alpha_i-\alpha_q+1)_l}}x^l.
\end{multline*}
Through the corresponding changes of variables we get
\begin{multline*}
 L_{\vec{n}}^{(i)}(x;c_1,\dots,c_p,\alpha_0)=\lim_{t\rightarrow\infty}\dfrac{(-1)^{|\vec{n}|-1}}{t^{\alpha_0+|\vec{n}|}}P^{(i)}_{\vec{n}}\left(1-\dfrac{x}{t};c_1t,\dots,c_pt,\alpha_0\right)\\
 =\lim_{t\rightarrow\infty}\dfrac{1}{(n_i-1)!\Gamma(\alpha_0+|\vec{n}|)}\dfrac{1}{t^{\alpha_0+|\vec{n}|}}\dfrac{\prod_{q=1}^{p}(c_qt+\alpha_0+|\vec{n}|)_{n_q}}{\prod_{q=1,q\neq i}^{p}(c_qt-c_it)_{n_q}} \dfrac{\Gamma(c_it+\alpha_0+|\vec{n}|)}{\Gamma(c_it+1)}\\
 \times\sum_{l=0}^{n_i-1}
 \dfrac{(-n_i+1)_l}{l!}
 {\dfrac{(c_it+\alpha_0+|\vec{n}|)_l}{(c_it+1)_l}}
 \prod_{q=1,q\neq i}^p
 {\dfrac{(c_it-c_qt-n_q+1)_l}{(c_it-c_qt+1)_l}}
 {\left(1-\dfrac{x}{t}\right)^l}.
\end{multline*}
We can use now Gauss's hypergeometric formula
\begin{align*}
\begin{aligned}
	 \pFq{2}{1}{-n,b}{c}{1}&=\sum_{l=0}^{n}\dfrac{(-n)_l}{l!}\dfrac{(b)_l}{(c)_l}=\dfrac{(c-b)_n}{(c)_n},& n&\in\mathbb N_0,
\end{aligned}
\end{align*}
to write the previous fractions within the sum as
\begin{align*}
{\dfrac{(c_it+\alpha_0+|\vec{n}|)_l}{(c_it+1)_l}}&={\sum_{k_i=0}^{l}\frac{(-l)_{k_i}(-\alpha_0-|\vec{n}|+1)_{k_i}}{k_i!(c_it+1)_{k_i}}},\\
{\dfrac{(c_it-c_qt-n_q+1)_l}{(c_it-c_qt+1)_l}}&={\sum_{k_q=0}^{l}\frac{(-l)_{k_q}(n_q)_{k_q}}{k_q!(c_it-c_qt+1)_{k_q}}},
\end{align*}
which leads to
\begin{multline*}
 L_{\vec{n}}^{(i)}(x;c_1,\dots,c_p,\alpha_0)
=\lim_{t\rightarrow\infty}\dfrac{1}{(n_i-1)!\Gamma(\alpha_0+|\vec{n}|)}\dfrac{1}{t^{\alpha_0+|\vec{n}|}}\dfrac{\prod_{q=1}^{p}(c_qt+\alpha_0+|\vec{n}|)_{n_q}}{\prod_{q=1,q\neq i}^{p}(c_qt-c_it)_{n_q}} \dfrac{\Gamma(c_it+\alpha_0+|\vec{n}|)}{\Gamma(c_it+1)}\\
 \times\sum_{l=0}^{n_i-1}
 \dfrac{(-n_i+1)_l}{l!}
 {\sum_{k_i=0}^{l}\frac{(-l)_{k_i}(-\alpha_0-|\vec{n}|+1)_{k_i}}{k_i!(c_it+1)_{k_i}}}
 \prod_{q=1,q\neq i}^p
{\sum_{k_q=0}^{l}\frac{(-l)_{k_q}(n_q)_{k_q}}{k_q!(c_it-c_qt+1)_{k_q}}}
 {\left(1-\dfrac{x}{t}\right)^l}.
\end{multline*}
Note that the factors outside the sum behave asymptotically for $t\to\infty$ as
\begin{align*}
 \dfrac{c_i^{\alpha_0+|\vec{n}|-1}}{(n_i-1)!\Gamma(\alpha_0+|\vec{n}|)}\dfrac{\prod_{q=1}^{p}c_q^{n_q}}{\prod_{q=1,q\neq i}^{p}(c_q-c_i)^{n_q}}\,t^{n_i-1}+O\left(t^{n_i-2}\right),
\end{align*}
while the multiple sum has an asymptotic expansion for $t\to\infty$:
\begin{multline*}
 \sum_{k_1=0}^{n_i-1}\cdots\sum_{k_p=0}^{n_i-1-k_1-\cdots-k_{p-1}}
 \dfrac{1}{k_1!\cdots k_p!}\dfrac{(-\alpha_0-|\vec{n}|+1)_{k_i}}{c_i^{k_i}t^{k_i}}
 \\ \times \prod_{q=1,q\neq i}^p\dfrac{(n_q)_{k_q}}{(c_i-c_q)^{k_q}t^{k_q}}
 \sum_{l=0}^{n_i-1}\dfrac{(-n_i+1)_l}{l!}(-l)_{k_1+\cdots+k_p}\left(1-\dfrac{x}{t}\right)^l+ O\left(\dfrac{1}{t^{n_i}}\right).
\end{multline*}
The Pochhammer product $(-l)_{k_1}\cdots(-l)_{k_p}$ equals $(-l)_{k_1+\cdots+k_p}$ plus less order Pochhammer symbols, whose contributions have been included within $O\left(\frac{1}{t^{n_i}}\right)$.

Now Newton's binomial formula leads to:
\begin{align*}
 \sum_{l=0}^{n_i-1}&\dfrac{(-n_i+1)_l}{l!}(-l)_{k_1+\cdots+k_p}\left(1-\dfrac{x}{t}\right)^l\\
 &= (-1)^{k_1+\cdots+k_p}(-n_i+1)_{k_1+\cdots+k_p}\left(1-\dfrac{x}{t}\right)^{k_1+\cdots+k_p}\sum_{l=0}^{n_i-1-k_1-\cdots-k_p}\dbinom{n_i-1-k_1-\cdots-k_p}{l}\left(\dfrac{x}{t}-1\right)^{l}\\
 &=(-1)^{k_1+\cdots+k_p}(-n_i+1)_{k_1+\cdots+k_p}\left(1-\dfrac{x}{t}\right)^{k_1+\cdots+k_p}\left(\dfrac{x}{t}\right)^{n_i-1-k_1-\cdots-k_p}.
\end{align*}
Replacing this sum in the previous expression we find the following asymptotic expansion for $t\to\infty$
\begin{multline*}
 \dfrac{1}{t^{n_i-1}}\sum_{k_1=0}^{n_i-1}\cdots\sum_{k_p=0}^{n_i-1-k_1-\cdots-k_{p-1}}\dfrac{(-n_i+1)_{k_1+\cdots+k_p}}{k_1!\cdots k_p!}\dfrac{(-\alpha_0-|\vec{n}|+1)_{k_i}}{(-c_i)^{k_i}}\\\times \prod_{q=1,q\neq i}^p\dfrac{(n_q)_{k_q}}{(c_q-c_i)^{k_q}}
 \left(1-\dfrac{x}{t}\right)^{k_1+\cdots+k_p}{x}^{n_i-1-k_1-\cdots-k_p}+ O\left(\dfrac{1}{t^{n_i}}\right).
\end{multline*}
Combining the common factor with the sums we find that the whole expression behaves asymptotically for $t\to\infty$ as follows:
\begin{multline*}
\dfrac{c_i^{\alpha_0+|\vec{n}|-1}}{(n_i-1)!\Gamma(\alpha_0+|\vec{n}|)}\dfrac{\prod_{q=1}^{p}c_q^{n_q}}{\prod_{q=1,q\neq i}^{p}(c_q-c_i)^{n_q}}
\sum_{k_1=0}^{n_i-1}\cdots\sum_{k_p=0}^{n_i-1-k_1-\cdots-k_{p-1}}\dfrac{(-n_i+1)_{k_1+\cdots+k_p}}{k_1!\cdots k_p!}\dfrac{(-\alpha_0-|\vec{n}|+1)_{k_i}}{(-c_i)^{k_i}}\\\times\prod_{q=1,q\neq i}^p\dfrac{(n_q)_{k_q}}{(c_q-c_i)^{k_q}}
x^{n_i-1-k_1-\cdots-k_p}\left(1-\dfrac{x}{t}\right)^{k_1+\cdots+k_p}+ O\left(\dfrac{1}{t}\right).
\end{multline*}
Now, is trivial to apply the limit $t\rightarrow\infty$  finding that 
\begin{multline*}
L_{\vec{n}}^{(i)}(x;c_1,\dots,c_p,\alpha_0)
=\dfrac{c_i^{\alpha_0+|\vec{n}|-1}}{(n_i-1)!\Gamma(\alpha_0+|\vec{n}|)}\dfrac{\prod_{q=1}^{p}c_q^{n_q}}{\prod_{q=1,q\neq i}^{p}(c_q-c_i)^{n_q}}\\
\times\sum_{k_1=0}^{n_i-1}\cdots\sum_{k_p=0}^{n_i-1-k_1-\cdots-k_{p-1}}\dfrac{(-n_i+1)_{k_1+\cdots+k_p}}{k_1!\cdots k_p!}\dfrac{(-\alpha_0-|\vec{n}|+1)_{k_i}}{(-c_i)^{k_i}}\prod_{q=1,q\neq i}^p\dfrac{(n_q)_{k_q}}{(c_q-c_i)^{k_q}}
x^{n_i-1-k_1-\cdots-k_p}.
\end{multline*}
A more convenient form of this limit one appears by the index changes $k_1\rightarrow n_i-1-l_1$, $k_i\rightarrow l_{i-1}-l_i$, $i\in\{2,\ldots,p\}$. Hence, the sums transforms to desired expression \eqref{LaguerreSKTypeI}.
\end{proof}

For $p=2$ and $i\in\{1,2\}$, the previous expression reads
\begin{align*}
 L_{(n_1,n_2)}^{(i)}(x;c_1,c_2,\alpha_0){=}&\dfrac{(-1)^{{n}_i-1}(\alpha_0+\hat{n}_i+1)_{n_i-1}}{(n_i-1)!\Gamma(\alpha_0+n_1+n_2)}c_i^{\alpha_0+n_1+n_2}\left(\dfrac{\hat{c}_i}{\hat{c}_i-c_i}\right)^{\hat{n}_i}\KF{1:0;1}{1:0;0}{-n_i+1:--;\hat{n}_i}{\alpha_0+\hat{n}_i+1:--;--}{c_ix,\dfrac{c_i}{c_i-\hat{c}_i}},
\end{align*}
where, $\hat{c}_i\coloneq\delta_{i,1}c_2+\delta_{i,2}c_1$ and $\hat{n}_i\coloneq\delta_{i,1}n_2+\delta_{i,2}n_1=n_1+n_2-n_i$ . 
This is not the expression found at \cite[\S 8]{HahnI}. This leads us to the following hypergeometric formula connecting Kampé de Fériet series.

\begin{cor}
Let be $n_i,\hat{n}_i\in\mathbb N$; $x\in\mathbb R^+$; $\alpha_0>-1$ and $c_i,\hat{c}_i>0$ with $c_i\neq\hat{c}_i$ then 
 \begin{multline*}
 \KF{1:0;1}{1:0;0}{-n_i+1:--;\hat{n}_i}{\alpha_0+\hat{n}_i+1:--;--}{c_ix,\dfrac{c_i}{c_i-\hat{c}_i}}\\
 \\=
 \dfrac{(\hat{n}_i)_{n_i-1}}{(\alpha_0+\hat{n}_i+1)_{n_i-1}}\left(\dfrac{\hat{c}_i}{\hat{c}_i-c_i}\right)^{n_i-1}\KF{2:0;0}{1:1;0}{-n_i+1,\,\alpha_0+1:--;--}{-n_1-n_2+2:\alpha_0+1;--}{\dfrac{(c_i-\hat{c}_i)c_i}{\hat{c}_i}x,-\dfrac{(c_i-\hat{c}_i)}{\hat{c}_i}}.
\end{multline*}
\end{cor}

\subsection{Explicit Expressions for the Recurrence Coefficients}\label{S.L2Krecurrence}

Now, we will proceed to determine the corresponding recurrence coefficients:
\begin{teo}
 The Laguerre of the second kind multiple orthogonal polynomials, both type I as described in \eqref{LaguerreSKTypeI} and type II as in \eqref{LaguerreSKTypeII}, adhere to their respective nearest-neighbor recurrence relations, as depicted in \eqref{NextNeighbourRecurrence}, with respect to the coefficients,
\begin{align}
\label{LaguerreSKRecurrence}
 \begin{aligned}
 	 b^0_{\vec{n}}(k)&=\dfrac{\alpha_0+|\vec{n}|+1}{c_k}+\sum_{i=1}^p\dfrac{n_i}{c_i},\\
b^j_{\vec{n}}&=(-1)^{j+1}(\alpha_0+|\vec{n}|-j+1)_{j}\sum_{i\in 
S(\pi,j)}\dfrac{n_i}{c_i^{j+1}}\prod_{q\in S^{\textsf c}(\pi,j)}\dfrac{c_i-c_q}{c_q},& j&\in\{1,\dots,p\}.
\end{aligned}
\end{align}
\end{teo}

\begin{proof}

Utilizing either of the limits provided in the type I case \eqref{LaguerreSKasJPLimitTypeI} or type II case \eqref{LaguerreSKTypeII} over the corresponding nearest-neighbor recurrence relation \eqref{NextNeighbourRecurrence} for the Jacobi–Piñeiro polynomials, we deduce that the Laguerre coefficients can be derived from the Jacobi–Piñeiro coefficients outlined in \eqref{JPRecurrence} through the following limit:
\begin{align*}
 &\lim_{t\rightarrow\infty}(-t)\left(b^0_{\vec{n}}(k;c_1t,\dots,c_pt,\alpha_0)-1\right),\\
 &\begin{aligned}
 	&\lim_{t\rightarrow\infty}(-t)^{j+1}b^j_{\vec{n}}(c_1t,\dots,c_pt,\alpha_0),& j&\in\{1,\dots,p\}.
 \end{aligned}
\end{align*}

For the case 
where
$j>0$, applying this limit is straightforward and yields the previously mentioned expression.
Let's take a moment to address the case 
when
$j=0$.
Here, after the variable changes, the recurrence coefficients from \eqref{JPRecurrence} remain unchanged:
\begin{gather*}
 \begin{multlined}[t][.67\textwidth]
 	 \dfrac{(n_k+1)(c_kt+n_k+1)}{(c_kt+\alpha_0+n_k+|\vec{n}|+2)}\prod_{q=1,q\neq k}^p\dfrac{(c_k-c_q)t+n_k+1}{(c_k-c_q)t+n_k+1-n_q}
 -\dfrac{n_k(c_kt+n_k)}{(c_kt+\alpha_0+n_k+|\vec{n}|)}\prod_{q=1,q\neq k}^p\dfrac{(c_k-c_q)t+n_k}{(c_k-c_q)t+n_k-n_q}\\
 +\sum_{i=1,i\neq k}^p \dfrac{n_i(c_it+n_i)(c_kt+\alpha_0+n_k+|\vec{n}|+1)}{(c_it+\alpha_0+n_i+|\vec{n}|)_2((c_i-c_k)t-n_k+n_i-1)}\prod_{q=1,q\neq i}^p\dfrac{(c_i-c_q)t+n_i}{(c_i-c_q)t+n_i-n_q}
 \end{multlined}\\
 \begin{multlined}[t][.8\textwidth]
 =(n_k+1)\left(1-\dfrac{\alpha_0+|\vec{n}|+1}{c_kt+\alpha_0+n_k+|\vec{n}|+2}\right)\prod_{q=1,q\neq k}^p\left(1+\dfrac{n_q}{(c_k-c_q)t+n_k+1-n_q}\right)\\
 -n_k\left(1-\dfrac{\alpha_0+|\vec{n}|}{c_kt+\alpha_0+n_k+|\vec{n}|}\right)\prod_{q=1,q\neq k}^p\left(1+\dfrac{n_q}{(c_k-c_q)t+n_k-n_q}\right)\\
 +\sum_{i=1,i\neq k}^p \dfrac{n_i(c_it+n_i)(c_kt+\alpha_0+n_k+|\vec{n}|+1)}{(c_it+\alpha_0+n_i+|\vec{n}|)_2((c_i-c_k)t-n_k+n_i-1)}\prod_{q=1,q\neq i}^p\dfrac{(c_i-c_q)t+n_i}{(c_i-c_q)t+n_i-n_q}.
 \end{multlined}
\end{gather*}
The components within the sum on the right-hand side behave for $t\to\infty$ as:
\begin{align*}
 \dfrac{n_ic_k}{c_i(c_i-c_k)}\dfrac{1}{t}+O\left(\dfrac{1}{t^2}\right),
\end{align*}
while the other two summands behave for $t\to\infty$ as follows:
\begin{align*}
 &(n_k+1)\left(1+\dfrac{1}{t}\sum_{q=1,q\neq k}^p\dfrac{n_q}{c_k-c_q}-\dfrac{\alpha_0+|\vec{n}|+1}{tc_k}\right)+O\left(\dfrac{1}{t^2}\right),\\
 &n_k\left(1+\dfrac{1}{t}\sum_{q=1,q\neq k}^p\dfrac{n_q}{c_k-c_q}-\dfrac{\alpha_0+|\vec{n}|}{tc_k}\right)+ O\left(\dfrac{1}{t^2}\right),
\end{align*}
respectively.

Combining all the summands and simplifying, we find that the entire expression behaves for $t\to\infty$ as:
\begin{align*}
 1-\dfrac{1}{t}\left(\dfrac{\alpha_0+|\vec{n}|+1}{c_k}+\sum_{i=1}^p\dfrac{n_i}{c_i} \right)+ O\left(\dfrac{1}{t^2}\right).
\end{align*}
The limit now becomes immediate to apply, resulting in \eqref{LaguerreSKRecurrence}.
\end{proof}

In the step line this expression becomes:
\begin{cor}[Step line Laguerre of the second kind recurrence coefficients]
The Laguerre of second kind type I 
and II \eqref{LaguerreSKTypeII} multiple orthogonal polynomials in the step line satisfy a recurrence relation of the form \eqref{StepLineRecurrence} respect to the coefficients
\begin{align*}
	b^0_{pm+k}&=\dfrac{\alpha_0+pm+k+1}{c_{k+1}}+\sum_{i=1}^k\dfrac{m+1}{c_i}+\sum_{i=k+1}^p\dfrac{m}{c_i},
	\end{align*}
	and
	\begin{align*}
b^j_{pm+k}&=(-1)^{j+1}(\alpha_0+pm+k-j+1)_{j} 
\left(\sum_{i=1}^{k+1-j}\dfrac{m+1}{c_i^{j+1}}\prod_{q=k+2-j}^{k}\dfrac{c_i-c_q}{c_q}+\sum_{i=k+1}^{p}\dfrac{m}{c_i^{j+1}}\prod_{q=k+2-j}^{k}\dfrac{c_i-c_q}{c_q}\right),
\end{align*}
 for $ j\in\{1,\dots,k\},$ while for $ j\in\{k+1,\dots,p\}$

\begin{align*}
	b^j_{pm+k}=(-1)^{j+1}(\alpha_0+pm+k-j+1)_{j}\sum_{i=k+1}^{p+k+1-j}\dfrac{m}{c_i^{j+1}}\prod_{q=1}^{k}\dfrac{c_i-c_q}{c_q}
\prod_{q=p+k+2-j}^{p}\dfrac{c_i-c_q}{c_q}.
\end{align*}
\end{cor}

\section{Multiple Hermite with $p$ Weights: Type I Polynomials and the Recurrence Coefficients}

The weight functions for this family are
\begin{align}
 \label{WeightsHermite}
 & \begin{aligned}
 	 w_i(x;c_i)&\coloneq \exp{(-x^2+c_ix)}, & i&\in\{1,\dots,p\}, 
 \end{aligned}& \d\mu(x)&=\d x,& \Delta&=\mathbb R,
\end{align}
with \( c_1, \ldots, c_p \in \mathbb{R} \), and for an AT system, \( c_i \neq c_j \) for \( i \neq j \). The corresponding monic type II polynomials are as follows (cf. \cite[\S 23.5]{Ismail}):
\begin{align}
 \label{HermiteTypeII}
 H_{\vec{n}}(x;c_1,\ldots,c_p)&=\left(-\dfrac{1}{2}\right)^{|\vec{n}|}\prod_{q=1}^pc_q^{n_q}\sum_{l_1=0}^{n_1}\cdots\sum_{l_p=0}^{n_p}2^{l_1+\cdots+l_p}\prod_{q=1}^p\dfrac{(-n_q)_{l_q}}{l_q!c_q^{l_q}}\,H_{l_1+\cdots+l_p}(x) ,
\end{align}
where \( H_n \) represents the usual monic  Hermite polynomials,
\begin{align}
 \label{Hermite}
 H_n(x)=\sum_{k=0}^{\left\lfloor\frac{n}{2}\right\rfloor} (-1)^{k}\dfrac{(-n)_{2k}}{k!}\dfrac{x^{n-2k}}{2^{2k}}, && n \in \mathbb N_0 .
\end{align}
These can be obtained through the respective limits from the Jacobi–Piñeiro type II polynomials \eqref{JPTypeII} and the Laguerre of the first kind type II polynomials \eqref{LaguerreFKTypeII}:
\begin{align}
\label{HermiteasJPLimitTypeII}
 H_{\vec{n}}(x;c_1,\ldots,c_p)&=\lim_{\beta\rightarrow\infty}\left(2\sqrt{\beta}\right)^{|\vec{n}|}P_{\vec{n}}\left(\dfrac{x+\sqrt{\beta}}{2\sqrt{\beta}};\beta+c_1\sqrt{\beta},\dots,\beta+c_p\sqrt{\beta},\beta\right),\\
 \label{HermiteasLaguerreFKLimitTypeII}
 H_{\vec{n}}(x;c_1,\ldots,c_p)&=\lim_{\beta\rightarrow\infty}\dfrac{1}{\left(\sqrt{2\beta}\right)^{|\vec{n}|}}\,L_{\vec{n}}\left(\sqrt{2\beta}x+\beta;\beta+c_1\sqrt{\dfrac{\beta}{2}},\dots,\beta+c_p\sqrt{\dfrac{\beta}{2}}\right).
\end{align}

\subsection{Explicit Expressions for the Type I Polynomials}\label{S.Hermite.TypeI}

In \cite{HahnI}, we encountered difficulty in finding an explicit expression for the Hermite type I polynomials for a system with \( p = 2 \) weight functions. Nonetheless, here, we will advance further and demonstrate the following expression for a general number \( p \geq 2 \) of weights:
\begin{teo}
The Hermite multiple orthogonal polynomials of type I are
\begin{multline}
\label{HermiteTypeI}
 H^{(i)}_{\vec{n}}(x;c_1,\dots,c_p)= {\dfrac{ (-1)^{n_i-1}}{\sqrt{\pi}(n_i-1)!}\dfrac{2^{|\vec{n}|-1}\exp{\left(-\frac{c_i^2}{4}\right)}}{\prod_{q=1,q\neq i}^{p}(c_i-c_q)^{n_q}} }\\
 \times\sum_{l_1=0}^{n_i-1}\cdots\sum_{l_p=0}^{n_i-1-l_1-\cdots-l_{p-1}}
 \dfrac{(-n_i+1)_{l_1+\cdots+l_p}}{l_i!} \prod_{q=1,q\neq i}^p\dfrac{(n_q)_{l_q}}{l_q!(c_q-c_i)^{l_q}}{H_{n_i-1-l_1-\cdots-l_p}\left(\dfrac{c_i}{2}\right)}
\,{x}^{l_i} ,
\end{multline}
for \( i \in \{1,\dots,p\} \), where \( H_n \) represents the standard Hermite polynomials as defined in \eqref{Hermite}. These can be derived from the Jacobi–Piñeiro polynomials \eqref{JPTypeI} through the limit:
\begin{align}
\label{HermiteasJPLimitTypeI}
 H_{\vec{n}}^{(i)}(x,c_1,\dots,c_p)=\lim_{\beta\rightarrow\infty}\dfrac{1}{\left(2\sqrt{\beta}\right)^{|\vec{n}|}2^{2\beta+c_i\sqrt{\beta}}}\,P_{\vec{n}}^{(i)}\left(\dfrac{x+\sqrt{\beta}}{2\sqrt{\beta}},\beta+c_1\sqrt{\beta},\dots,\beta+c_p\sqrt{\beta},\beta\right),
\end{align}
or from the Laguerre of the first kind polynomials \eqref{LaguerreFKTypeI} through the limit:
\begin{align}
 \label{HermiteasLaguerreFKLimitTypeI}
 H_{\vec{n}}^{(i)}(x;c_1,\dots,c_p)={\lim_{\beta\to\infty} {\left(\sqrt{2\beta}\right)^{|\vec{n}|}\beta^{\,\beta+c_i\sqrt{\frac{\beta}{2}}}\Exp{\,-\beta}}\,L_{\vec{n}}^{(i)}\left({\sqrt{2\beta}\,x+\beta},\beta+c_1\sqrt{\dfrac{\beta}{2}},\dots,\beta+c_p\sqrt{\dfrac{\beta}{2}}\right)}.
\end{align}
\end{teo}

\begin{proof}

As in the previous section, we will divide the proof into two parts. Firstly, we will demonstrate that the polynomials defined by both previous limits satisfy the orthogonality conditions with respect to the weight functions \eqref{WeightsHermite}, as outlined in \eqref{ortogonalidadTipoIContinua}. Let's begin with the Jacobi–Piñeiro limit.

The Jacobi–Piñeiro type I polynomials adhere to the orthogonality conditions:
\begin{align*}
 \sum_{i=1}^p\int_{0}^1 \left(x-\dfrac{1}{2}\right)^k P_{\vec{n}}^{(i)}(x;\alpha_1,\dots,\alpha_p,\beta) x^{\alpha_i}(1-x)^{\beta}\d x=\begin{cases}
 0,\;\text{if}\; k\in\{0,\dots,|\vec{n}|-2\},\\
 1,\;\text{if}\; k=|\vec{n}|-1.
 \end{cases}
\end{align*}
Under the change of variables
\begin{align*}
\begin{aligned}
	 x&\rightarrow \dfrac{x+\sqrt{\beta}}{2\sqrt{\beta}},& \alpha_i&\rightarrow \beta+c_i\sqrt{\beta} ,
\end{aligned}
\end{align*}
the previous expression transforms into:
\begin{multline*}
 \sum_{i=1}^p
 \int_{-\sqrt{\beta}}^{\sqrt{\beta}} x^k \dfrac{1}{\left(2\sqrt{\beta}\right)^{|\vec{n}|}2^{2\beta+c_i\sqrt{\beta}}}\,P_{\vec{n}}^{(i)}\left(\dfrac{x+\sqrt{\beta}}{2\sqrt{\beta}},\beta+c_1\sqrt{\beta},\dots,\beta+c_p\sqrt{\beta},\beta\right){\left(1-\dfrac{x^2}{\beta}\right)^{\beta}\left(1+\dfrac{x}{\sqrt{\beta}}\right)^{c_i\sqrt{\beta}}}\d x\\=\begin{cases}
 0,\;\text{if}\;k\in\{0,\dots,|\vec{n}|-2\},\\
 1,\;\text{if}\; k=|\vec{n}|-1.
 \end{cases}
\end{multline*}

Now, our goal is to apply the limit as \( \beta \rightarrow \infty \) to both sides of the preceding equation. To ensure that the limit can be interchanged with the integral, we need to establish that the modulus of the expression within the integral is bounded by an integrable function for all \( \beta > 0 \). On one hand, we have:
\begin{align*}
\begin{aligned}
	 {\left(1-\dfrac{x^2}{\beta}\right)^{\beta}\left(1+\dfrac{x}{\sqrt{\beta}}\right)^{c_i\sqrt{\beta}}}I_{\left(-\sqrt{\beta},\sqrt{\beta}\right)}&\leq \exp{(-x^2+c_i x)},&\beta&>0, & x&\in\mathbb R.
\end{aligned}
\end{align*}
On the other hand, the polynomial
\begin{align*}
\dfrac{1}{\left(2\sqrt{\beta}\right)^{|\vec{n}|}2^{2\beta+c_i\sqrt{\beta}}}\,P_{\vec{n}}^{(i)}\left(\dfrac{x+\sqrt{\beta}}{2\sqrt{\beta}},\beta+c_1\sqrt{\beta},\dots,\beta+c_p\sqrt{\beta},\beta\right) ,
\end{align*}
is bounded for all \( \beta > 0 \) since the limit \( \beta \to \infty \) exists, as we will demonstrate shortly.

Thus, the expression is bounded, and by Lebesgue's dominated convergence theorem, we can interchange the limit to obtain:
\begin{align*}
 \sum_{i=1}^p
 &\int_{-\infty}^{\infty} x^k \underbrace{\lim_{\beta\rightarrow\infty}\dfrac{1}{\left(2\sqrt{\beta}\right)^{|\vec{n}|}2^{2\beta+c_i\sqrt{\beta}}}\,P_{\vec{n}}^{(i)}\left(\dfrac{x+\sqrt{\beta}}{2\sqrt{\beta}},\beta+c_1\sqrt{\beta},\dots,\beta+c_p\sqrt{\beta},\beta\right)}_{=H^{(i)}_{\vec{n}}(x;c_1,\dots,c_p)\;\text{by definition}}
 \underbrace{\lim_{\beta\rightarrow\infty}\left(1-\dfrac{x^2}{\beta}\right)^{\beta}\left(1+\dfrac{x}{\sqrt{\beta}}\right)^{c_i\sqrt{\beta}}}_{=\exp{(-x^2+c_ix)}}\d x\\&=\begin{cases}
 0,\;\text{if}\;k\in\{0,\dots,|\vec{n}|-2\},\\
 1,\;\text{if}\; k=|\vec{n}|-1.
 \end{cases}
\end{align*}

The reasoning from the Laguerre limit follows a completely analogous path. The Laguerre of the first kind type I polynomials adhere to the orthogonality conditions:
\begin{align*}
 \sum_{i=1}^p\int_{0}^\infty \left(x-\beta\right)^k L_{\vec{n}}^{(i)}(x;\alpha_1,\dots,\alpha_p) x^{\alpha_i}\Exp{-x}\d x=\begin{cases}
 0,\;\text{if}\; k\in\{0,\dots,|\vec{n}|-2\},\\
 1,\;\text{if}\; k=|\vec{n}|-1.
 \end{cases}
\end{align*}
Under the change of variables
\begin{align*}
\begin{aligned}
	 x&\rightarrow {\sqrt{2\beta}\,x+\beta},& \alpha_i&\rightarrow \beta+c_i\sqrt{\dfrac{\beta}{2}} ,
\end{aligned}
\end{align*}
the previous expression becomes
\begin{align*}
 \sum_{i=1}^p
 \int_{-\sqrt{\frac{\beta}{2}}}^{\infty}\; &x^k {\left(\sqrt{2\beta}\right)^{|\vec{n}|}\beta^{\,\beta+c_i\sqrt{\frac{\beta}{2}}}\Exp{\,-\beta}}\,L_{\vec{n}}^{(i)}\left({\sqrt{2\beta}\,x+\beta},\beta+c_1\sqrt{\dfrac{\beta}{2}},\dots,\beta+c_p\sqrt{\dfrac{\beta}{2}}\right)\\
 &\times\left(1+\sqrt{\dfrac{2}{\beta}}\,x\right)^{c_i\sqrt{\frac{\beta}{2}}}\left(1+\sqrt{\dfrac{2}{\beta}}\,x\right)^{\beta}\exp{\left(-\sqrt{2\beta}\,x\right)}\d x=\begin{cases}
 0,\;\text{if}\;k\in\{0,\dots,|\vec{n}|-2\},\\
 1,\;\text{if}\; k=|\vec{n}|-1.
 \end{cases}
\end{align*}

Let's expand
\begin{align*}
 \left(1+\sqrt{\dfrac{2}{\beta}}\,x\right)^{\beta}=\exp{\left(\beta\ln{\left(1+\sqrt{\frac{2}{\beta}}\,x\right)}\right)}=\exp{\left(\beta\sum_{k=1}^{\infty}\frac{(-1)^{k+1}}{k}\sqrt{\frac{2}{\beta}}^k x^k\right)}=\exp{\left(\sqrt{{2}{\beta}}x- x^2+\sum_{k=3}^{\infty}\frac{(-1)^{k+1}}{k}\frac{\sqrt{2}^k}{\sqrt{\beta}^{k-2}} x^k\right)}
\end{align*}
to get
\begin{align*}
 \sum_{i=1}^p
 \int_{-\sqrt{\frac{\beta}{2}}}^{\infty}\; &x^k {\left(\sqrt{2\beta}\right)^{|\vec{n}|}\beta^{\,\beta+c_i\sqrt{\frac{\beta}{2}}}\Exp{\,-\beta}}\,L_{\vec{n}}^{(i)}\left({\sqrt{2\beta}\,x+\beta},\beta+c_1\sqrt{\dfrac{\beta}{2}},\dots,\beta+c_p\sqrt{\dfrac{\beta}{2}}\right)\\
 &\times\left(1+\sqrt{\dfrac{2}{\beta}}\,x\right)^{c_i\sqrt{\frac{\beta}{2}}}\exp{\left(- x^2+\sum_{k=3}^{\infty}\frac{(-1)^{k+1}}{k}\frac{\sqrt{2}^k}{\sqrt{\beta}^{k-2}} x^k\right)}\d x=\begin{cases}
 0,\;\text{if}\;k\in\{0,\dots,|\vec{n}|-2\},\\
 1,\;\text{if}\; k=|\vec{n}|-1.
 \end{cases}
\end{align*}

Now, our goal is to apply the limit as \( \beta \rightarrow \infty \) to both sides of the preceding equation. Similar to the previous case, we need to ensure that the modulus of the expression within the integral is bounded. On one hand, we have:
\begin{align*}
\begin{aligned}
	 \left(1+\sqrt{\dfrac{2}{\beta}}\,x\right)^{c_i\sqrt{\frac{\beta}{2}}}&\exp\left({- x^2+\sum_{k=3}^{\infty}\frac{(-1)^{k+1}}{k}\frac{\sqrt{2}^k}{\sqrt{\beta}^{k-2}} x^k}\right)I_{\left(-\sqrt{\frac{\beta}{2}},\infty\right)}\\&=\left(1+\sqrt{\dfrac{2}{\beta}}\,x\right)^{c_i\sqrt{\frac{\beta}{2}}}\left(1+\sqrt{\dfrac{2}{\beta}}\,x\right)^{\beta}\Exp{-\sqrt{2\beta}\,x}I_{\left(-\sqrt{\frac{\beta}{2}},\infty\right)}\leq \exp{\left(-x^2+c_i x\right)},&\beta&>0, & x&\in\mathbb R.
\end{aligned}
\end{align*}

On the other hand, the polynomial
\begin{align*}
{\left(\sqrt{2\beta}\right)^{|\vec{n}|}\beta^{\,\beta+c_i\sqrt{\frac{\beta}{2}}}\Exp{-\beta}}\,L_{\vec{n}}^{(i)}\left({\sqrt{2\beta}\,x+\beta},\beta+c_1\sqrt{\frac{\beta}{2}},\dots,\beta+c_p\sqrt{\frac{\beta}{2}}\right) 
\end{align*}
is bounded for all \( \beta > 0 \) due to the existence of the limit \( \beta \to \infty \).

Hence, the expression is bounded, and by Lebesgue's dominated convergence theorem, we can interchange the limit \( \beta\rightarrow\infty \) to obtain:
\begin{multline*}
	\sum_{i=1}^p
	\int_{-\infty}^{\infty}x^k\lim_{\beta\to\infty} {\left(\sqrt{2\beta}\right)^{|\vec{n}|}\beta^{\,\beta+c_i\sqrt{\frac{\beta}{2}}}\Exp{-\beta}}\,L_{\vec{n}}^{(i)}\left({\sqrt{2\beta}\,x+\beta},\beta+c_1\sqrt{\frac{\beta}{2}},\dots,\beta+c_p\sqrt{\frac{\beta}{2}}\right)\\
	\times\lim_{\beta\to\infty} {\left(1+\sqrt{\frac{2}{\beta}}\,x\right)^{c_i\sqrt{\frac{\beta}{2}}}\exp\left({- x^2+\sum_{k=3}^{\infty}\frac{(-1)^{k+1}}{k}\frac{\sqrt{2}^k}{\sqrt{\beta}^{k-2}} x^k}\right)}\,\d x
	=\begin{cases}
		0, &\text{if}\;k\in\{0,\dots,|\vec{n}|-2\},\\
		1, &\text{if}\; k=|\vec{n}|-1.
	\end{cases}
\end{multline*}
Once we have established that the polynomials defined by both of these limits are the Hermite polynomials, it is sufficient to compute either of them to obtain an explicit expression. Let's consider, for example, the Jacobi–Piñeiro limit \eqref{HermiteasJPLimitTypeI}.

Recall that the Jacobi–Piñeiro type I polynomials are:
\begin{multline*}
 P^{(i)}_{\vec{n}}(x;\alpha_1,\dots,\alpha_p,\beta) =(-1)^{|\vec{n}|-1}\dfrac{\prod_{q=1}^{p}(\alpha_q+\beta+|\vec{n}|)_{n_q}}{(n_i-1)!\prod_{q=1,q\neq i}^{p}(\alpha_q-\alpha_i)_{n_q}} \dfrac{\Gamma(\alpha_i+\beta+|\vec{n}|)}{\Gamma(\beta+|\vec{n}|)\Gamma(\alpha_i+1)}\\\times \sum_{l=0}^{n_i-1}{\dfrac{(-n_i+1)_l(\alpha_i+\beta+|\vec{n}|)_l \prod_{q=1,q\neq i}^p(\alpha_i-\alpha_q-n_q+1)_l}{l!(\alpha_i+1)_l\prod_{q=1,q\neq i}^p(\alpha_i-\alpha_q+1)_l}}x^l.
\end{multline*}
Through the corresponding changes of variables, we obtain:
\begin{multline*}
 \dfrac{1}{\left(2\sqrt{\beta}\right)^{|\vec{n}|}2^{2\beta+c_i\sqrt{\beta}}}P_{\vec{n}}^{(i)}\left(\dfrac{x+\sqrt{\beta}}{2\sqrt{\beta}},\beta+c_1\sqrt{\beta},\dots,\beta+c_p\sqrt{\beta},\beta\right)\\
 	={\dfrac{(-1)^{|\vec{n}|-1}}{(n_i-1)!\left(2\sqrt{\beta}\right)^{|\vec{n}|}2^{2\beta+c_i\sqrt{\beta}}}\,\dfrac{\prod_{q=1}^{p}(2\beta+c_q\sqrt{\beta}+|\vec{n}|)_{n_q}}{\prod_{q=1,q\neq i}^{p}(c_q\sqrt{\beta}-c_i\sqrt{\beta})_{n_q}} \dfrac{\Gamma(2\beta+c_i\sqrt{\beta}+|\vec{n}|)}{\Gamma(\beta+|\vec{n}|)\Gamma(\beta+c_i\sqrt{\beta}+1)}}\\
 \times\sum_{l=0}^{n_i-1}{\dfrac{(-n_i+1)_l}{l!}
 \dfrac{(2\beta+c_i\sqrt{\beta}+|\vec{n}|)_l}{(\beta+c_i\sqrt{\beta}+1)_l}\prod_{q=1,q\neq i}^p\dfrac{(c_i\sqrt{\beta}-c_q\sqrt{\beta}-n_q+1)_l}{(c_i\sqrt{\beta}-c_q\sqrt{\beta}+1)_l}}\,\left(\dfrac{x+\sqrt{\beta}}{2\sqrt{\beta}}\right)^l.
\end{multline*}

Let's reformulate the factors within the sum to a more convenient expression. Similar to the Laguerre of the second kind limit, we can utilize Gauss's hypergeometric formula:
\begin{align*}
\begin{aligned}
	 \pFq{2}{1}{-n,b}{c}{1}&=\sum_{l=0}^{n}\dfrac{(-n)_l}{l!}\dfrac{(b)_l}{(c)_l}=\dfrac{(c-b)_n}{(c)_n},& n&\in\mathbb N_0,
\end{aligned}
\end{align*}
to express the following fractions within the sum as:
\begin{align*}
{\dfrac{\left(c_i\sqrt{\beta}-c_q\sqrt{\beta}-n_q+1\right)_l}{\left(c_i\sqrt{\beta}-c_q\sqrt{\beta}+1\right)_l}}&={\sum_{k_q=0}^{l}\frac{(-l)_{k_q}(n_q)_{k_q}}{k_q!\left(c_i\sqrt{\beta}-c_q\sqrt{\beta}+1\right)_{k_q}}}.
\end{align*}
With the remaining Pochhammer fraction, we will need to engage in a bit more work. We will utilize the Gauss summation formula twice, as well as Newton's binomial formula for Pochhammer symbols
\begin{align*}
 (a+b)_n=\sum_{k=0}^n\dbinom{n}{k}(a)_{n-k}(b)_k.
\end{align*}
By employing these methods, we find that:
\begin{align*}
 {\dfrac{\left(2\beta+c_i\sqrt{\beta}+|\vec{n}|\right)_l}{\left
(\beta+c_i\sqrt{\beta}+1\right)_l}}&={\sum_{s=0}^{l}\frac{(-l)_{s}(-\beta-|\vec{n}|+1)_{s}}{s!\left
(\beta+c_i\sqrt{\beta}+1\right)_{s}}}={\sum_{s=0}^{l}(-1)^s\dfrac{(-l)_{s}(\beta+|\vec{n}|-s)_{s}}{s!\left
(\beta+c_i\sqrt{\beta}+1\right)_{s}}}\\
&={\sum_{s=0}^{l}\sum_{k_i=0}^{s}(-1)^s\dfrac{(-l)_{s}(-s)_{k_i}\left
(c_i\sqrt{\beta}-|\vec{n}|+1+s\right)_{k_i}}{s!k_i!\left
(\beta+c_i\sqrt{\beta}+1\right)_{k_i}}}\\
 &=\sum_{k_i=0}^{l}\sum_{s=0}^{l-k_i}{(-1)^{s+k_i}\dfrac{(-l)_{s+k_i}}{s!k_i!\left
(\beta+c_i\sqrt{\beta}+1\right)_{k_i}}}\underbrace{\left(-c_i\sqrt{\beta}+|\vec{n}|-2k_i-s\right)_{k_i}}_{=\sum_{k=0}^{k_i}\binom{k_i}{k}\left
(-c_i\sqrt{\beta}+|\vec{n}|-2k_i\right)_{k_i-k}(-s)_k}\\
 &={\sum_{k_i=0}^{l}\sum_{k=0}^{k_i}
 (-1)^{k_i}\dfrac{(-l)_{k_i+k}}{(k_i-k)!k!}\dfrac{\left
(-c_i\sqrt{\beta}+|\vec{n}|-2k_i\right)_{k_i-k}}{\left
(\beta+c_i\sqrt{\beta}+1\right)_{k_i}}\underbrace{\sum_{s=0}^{l-k_i-k}\dbinom{l-k_i-k}{s}}_{=2^{l-k_i-k}}}.
\end{align*}
So, substituting these expressions for the fractions within the multiple sum, we find:
\begin{multline*}
 H^{(i)}_{\vec{n}}(x;c_1,\dots,c_p)=\lim_{\beta\rightarrow\infty}{\dfrac{(-1)^{|\vec{n}|-1}}{(n_i-1)!\left(2\sqrt{\beta}\right)^{|\vec{n}|}2^{2\beta+c_i\sqrt{\beta}}}\,\dfrac{\prod_{q=1}^{p}(2\beta+c_q\sqrt{\beta}+|\vec{n}|)_{n_q}}{\prod_{q=1,q\neq i}^{p}(c_q\sqrt{\beta}-c_i\sqrt{\beta})_{n_q}} \dfrac{\Gamma(2\beta+c_i\sqrt{\beta}+|\vec{n}|)}{\Gamma(\beta+|\vec{n}|)\Gamma(\beta+c_i\sqrt{\beta}+1)}}\\
 \times\sum_{l=0}^{n_i-1}\dfrac{(-n_i+1)_l}{l!}{\sum_{k_i=0}^{l}\sum_{k=0}^{k_i}
 (-1)^{k_i}\dfrac{(-l)_{k_i+k}\,2^{l-k_i-k}}{(k_i-k)!k!}\dfrac{\left
(-c_i\sqrt{\beta}+|\vec{n}|-2k_i\right)_{k_i-k}}{\left
(\beta+c_i\sqrt{\beta}+1\right)_{k_i}}} 
 \\\times \prod_{q=1,q\neq i}^p{\sum_{k_q=0}^{l}\frac{(-l)_{k_q}(n_q)_{k_q}}{k_q!\left(c_i\sqrt{\beta}-c_q\sqrt{\beta}+1\right)_{k_q}}}\,\left(\dfrac{x+\sqrt{\beta}}{2\sqrt{\beta}}\right)^l.
\end{multline*}
Utilizing the asymptotic behavior of the gamma function fraction as \( \beta \rightarrow \infty \), we have:
\begin{align*}
 \dfrac{\Gamma(2\beta+c_i\sqrt{\beta}+|\vec{n}|)}{\Gamma(\beta+|\vec{n}|)\Gamma(\beta+c_i\sqrt{\beta}+1)}\sim
 \dfrac{2^{2\beta+c_i\sqrt{\beta}+|\vec{n}|-1}}{\sqrt{\pi}\sqrt{\beta}}{\exp\left({-\frac{c_i^2(2\sqrt{\beta}-c_i)}{8\sqrt{\beta}}}\right)}.
\end{align*}
The common factor behaves when \( \beta \to \infty \) as follows:
\begin{align*}
 {\dfrac{(-2)^{|\vec{n}|-1}}{\sqrt{\pi}(n_i-1)!}\dfrac{1}{\prod_{q=1,q\neq i}^{p}(c_q-c_i)^{n_q}} {\exp\left({-\frac{c_i^2(2\sqrt{\beta}-c_i)}{8\sqrt{\beta}}}\right)}\sqrt{\beta}^{n_i-1}}.
\end{align*}
Upon performing the index change \( k_i \rightarrow k_i - k \), we find that the multiple sum possesses an asymptotic expansion 
when
 \( \beta \to \infty \):
\begin{multline*}
 \sum_{k_1=0}^{n_i-1}\cdots\sum_{k_p=0}^{n_i-1-k_1-\cdots-k_{p-1}}
 \dfrac{c_i^{k_i}}{2^{k_i}}\sum_{k=0}^{\left\lfloor\frac{k_i}{2}\right\rfloor}\dfrac{(-1)^{k}}{c_i^{2k}}\dfrac{1}{(k_i-2k)!k!}\dfrac{1}{\sqrt{\beta}^{k_i}} \prod_{q=1,q\neq i}^p\dfrac{(n_q)_{k_q}}{k_q!(c_i-c_q)^{k_q}\sqrt{\beta}^{k_q}}
 \\
 \times\sum_{l=0}^{n_i-1}\dfrac{(-n_i+1)_l}{l!}(-l)_{k_1+\cdots+k_p}\,2^l\left(\dfrac{x+\sqrt{\beta}}{2\sqrt{\beta}}\right)^l+ O\left(\dfrac{1}{\sqrt{\beta}^{n_i}}\right).
\end{multline*}
The Pochhammer product \((-l)_{k_1}\cdots(-l)_{k_p}\) can be expressed as \((-l)_{k_1+\cdots+k_p}\) plus terms involving lower order Pochhammer symbols, whose contributions are encapsulated within \(O\left(\frac{1}{\sqrt{\beta}^{n_i}}\right)\).

Applying Newton's binomial formula, we obtain:
\begin{align*}
 \sum_{l=0}^{n_i-1}&\dfrac{(-n_i+1)_l}{l!}(-l)_{k_1+\cdots+k_p}\,2^l\left(\dfrac{x+\sqrt{\beta}}{2\sqrt{\beta}}\right)^l\\
 &= (-1)^{k_1+\cdots+k_p}(-n_i+1)_{k_1+\cdots+k_p}\left(\dfrac{x+\sqrt{\beta}}{\sqrt{\beta}}\right)^{k_1+\cdots+k_p}\sum_{l=0}^{n_i-1-k_1-\cdots-k_p}\dbinom{n_i-1-k_1-\cdots-k_p}{l}\left(-\dfrac{x+\sqrt{\beta}}{\sqrt{\beta}}\right)^{l}\\
 &=(-1)^{n_i-1}(-n_i+1)_{k_1+\cdots+k_p}\left(\dfrac{x+\sqrt{\beta}}{\sqrt{\beta}}\right)^{k_1+\cdots+k_p}\left(\dfrac{x}{\sqrt{\beta}}\right)^{n_i-1-k_1-\cdots-k_p}.
\end{align*}
Substituting this sum into the previous expression, we derive the following asymptotic expansion 
when
 \( \beta \to \infty \):
\begin{multline*}
 \dfrac{(-1)^{n_i-1}}{\sqrt{\beta}^{n_i-1}}\sum_{k_1=0}^{n_i-1}\cdots\sum_{k_p=0}^{n_i-1-k_1-\cdots-k_{p-1}}
 (-n_i+1)_{k_1+\cdots+k_p}\,\dfrac{c_i^{k_i}}{2^{k_i}}\sum_{k=0}^{\left\lfloor\frac{k_i}{2}\right\rfloor}\dfrac{(-1)^{k}}{c_i^{2k}}\dfrac{1}{(k_i-2k)!k!}\\
 \times\prod_{q=1,q\neq i}^p\dfrac{(n_q)_{k_q}}{k_q!(c_i-c_q)^{k_q}}\,{x}^{n_i-1-k_1-\cdots-k_p}+ O\left(\dfrac{1}{\sqrt{\beta}^{n_i}}\right).
\end{multline*}
Combining the common factor with the sums, we ascertain that the entire expression behaves asymptotically as follows for \( \beta \to \infty \):
\begin{multline*}
 \dfrac{ 2^{|\vec{n}|-1}}{\sqrt{\pi}(n_i-1)!}\dfrac{\exp\left({-\frac{c_i^2}{4}}\right)}{\prod_{q=1,q\neq i}^{p}(c_i-c_q)^{n_q}} \sum_{k_1=0}^{n_i-1}\cdots\sum_{k_p=0}^{n_i-1-k_1-\cdots-k_{p-1}}
 (-n_i+1)_{k_1+\cdots+k_p}\,\dfrac{c_i^{k_i}}{2^{k_i}}\sum_{k=0}^{\left\lfloor\frac{k_i}{2}\right\rfloor}\dfrac{(-1)^{k}}{c_i^{2k}}\dfrac{1}{(k_i-2k)!k!}\\
 \times\prod_{q=1,q\neq i}^p\dfrac{(n_q)_{k_q}}{k_q!(c_i-c_q)^{k_q}}\,{x}^{n_i-1-k_1-\cdots-k_p}+ O\left(\dfrac{1}{\sqrt{\beta}}\right).
\end{multline*}
Now, it is straightforward to apply the limit \( \beta \rightarrow \infty \) 
and
 find that:
\begin{multline*}
 H^{(i)}_{\vec{n}}(x;c_1,\dots,c_p)= {\dfrac{ 2^{|\vec{n}|-1}}{\sqrt{\pi}(n_i-1)!}\dfrac{\exp\left({-\frac{c_i^2}{4}}\right)}{\prod_{q=1,q\neq i}^{p}(c_i-c_q)^{n_q}} }\\
 \times\sum_{k_1=0}^{n_i-1}\cdots\sum_{k_p=0}^{n_i-1-k_1-\cdots-k_{p-1}}
 (-n_i+1)_{k_1+\cdots+k_p}\,\dfrac{c_i^{k_i}}{2^{k_i}}\sum_{k=0}^{\left\lfloor\frac{k_i}{2}\right\rfloor}\dfrac{(-1)^{k}}{c_i^{2k}}\dfrac{1}{(k_i-2k)!k!}\\\times \prod_{q=1,q\neq i}^p\dfrac{(n_q)_{k_q}}{k_q!(c_i-c_q)^{k_q}}\,{x}^{n_i-1-k_1-\cdots-k_p}.
\end{multline*}

A more convenient form of this limit emerges through the index changes \( k_1 \rightarrow n_i - 1 - l_1 \), \( k_i \rightarrow l_{i-1} - l_i \) for \( i \in \{2,\dots,p\} \). Hence, the sums transform to:
\begin{align*}
 H^{(i)}_{\vec{n}}(x;c_1,\dots,c_p)=& {\dfrac{ (-1)^{n_i-1}}{\sqrt{\pi}(n_i-1)!}\dfrac{2^{|\vec{n}|-1}\exp\left({-\frac{c_i^2}{4}}\right)}{\prod_{q=1,q\neq i}^{p}(c_i-c_q)^{n_q}} }\sum_{l_1=0}^{n_i-1}\cdots\sum_{l_p=0}^{n_i-1-l_1-\cdots-l_{p-1}}
 \dfrac{(-n_i+1)_{l_1+\cdots+l_p}}{l_i!} \prod_{q=1,q\neq i}^p\dfrac{(n_q)_{l_q}}{l_q!(c_q-c_i)^{l_q}}\\
 &\times\underbrace{\sum_{k=0}^{\left\lfloor\frac{n_i-1-l_1-\cdots-l_p}{2}\right\rfloor}(-1)^k\dfrac{(-n_i+1+l_1+\cdots+l_p)_{2k}}{k!}\dfrac{c_i^{n_i-1-l_1-\cdots-l_p-2k}}{2^{n_i-1-l_1-\cdots-l_p}}}_{=H_{n_i-1-l_1-\cdots-l_p}\left(\frac{c_i}{2}\right)\;\text{by \eqref{Hermite}}}
\,{x}^{l_i} .
\end{align*}
\end{proof}

\subsection{Explicit Expressions for the Recurrence Coefficients}
\label{S.Hermite.recurrence}

Now we are going to find the corresponding recurrence coefficients
\begin{teo}\label{teo:Hermite_recurrence} 
 The Hermite multiple orthogonal polynomials of type I \eqref{HermiteTypeI}
 and II \eqref{HermiteTypeII} satisfy respective nearest-neighbour recurrence relations of the form \eqref{NextNeighbourRecurrence} respect to the coefficients
\begin{align*}
	 b^0_{\vec{n}}(k)&=\dfrac{c_k}{2},&
\begin{aligned}b^j_{\vec{n}}&=\dfrac{1}{2^j}\sum_{i\in S(\pi,j)}n_i\prod_{q\in S^{\textsf c}(\pi,j)}{(c_i-c_q)},& j&\in\{1,\dots,p\}.
\end{aligned}
\end{align*}
\end{teo}

\begin{proof}
Both limits \eqref{HermiteasJPLimitTypeI} and \eqref{HermiteasJPLimitTypeII} imply that the Hermite coefficients can be obtained from the Jacobi–Piñeiro coefficients \eqref{JPRecurrence} through the limit:
\begin{align*}
\begin{aligned}
	 &\lim_{\beta\rightarrow\infty}\sqrt{\beta}\left(2\,b^0_{\vec{n}}\left(k;\beta+c_1\sqrt{\beta},\dots,\beta+c_p\sqrt{\beta},\beta\right)-{1}\right),\\
&\lim_{\beta\rightarrow\infty}\left(2\sqrt{\beta}\right)^{j+1}b^j_{\vec{n}}\left(\beta+c_1\sqrt{\beta},\dots,\beta+c_p\sqrt{\beta},\beta\right),&j&\in\{1,\dots,p\}.
\end{aligned}
\end{align*}
On the other hand, limits \eqref{HermiteasLaguerreFKLimitTypeII} and \eqref{HermiteasLaguerreFKLimitTypeI} imply that the Hermite coefficients can be obtained from the Laguerre of the first kind coefficients \eqref{LaguerreFKRecurrence} through the limit:
\begin{align*}
\begin{aligned}
	 &\lim_{\beta\rightarrow\infty}\dfrac{1}{\sqrt{2\beta}}\left(b^0_{\vec{n}}\left(k;\beta+c_1\sqrt{\dfrac{\beta}{2}},\dots,\beta+c_p\sqrt{\dfrac{\beta}{2}}\right)-\beta\right),\\
&\lim_{\beta\rightarrow\infty}\dfrac{1}{\left(\sqrt{2\beta}\right)^{j+1}} b^j_{\vec{n}}\left(\beta+c_1\sqrt{\dfrac{\beta}{2}},\dots,\beta+c_p\sqrt{\dfrac{\beta}{2}}\right),&j&\in\{1,\dots,p\}.
\end{aligned}
\end{align*}

In both cases, the limit for \( j \in \{1,\dots,p\} \) is straightforward to apply. Let's consider the Jacobi–Piñeiro limit for the case \( j = 0 \); the reasoning for the Laguerre limit is entirely analogous. Here, we observe that the recurrence coefficients \eqref{JPRecurrence} remain unchanged after the variable changes
\begin{align*}
 &	
 \begin{multlined}[t][.7\textwidth]
 2\,b^0_{\vec{n}}\left(k;\beta+c_1\sqrt{\beta},\dots,\beta+c_p\sqrt{\beta},\beta\right)
 ={(n_k+1)\dfrac{2\beta+2c_k\sqrt{\beta}+2n_k+2}{2\beta+c_k\sqrt{\beta}+n_k+|\vec{n}|+2}}\prod_{i=1,i\neq k}^p\dfrac{(c_k-c_i)\sqrt{\beta}+n_k+1}{(c_k-c_i)\sqrt{\beta}+n_k+1-n_i}\\
-{n_k\dfrac{2\beta+2c_k\sqrt{\beta}+2n_k}{2\beta+c_k\sqrt{\beta}+n_k+|\vec{n}|}}\prod_{i=1,i\neq k}^p\dfrac{(c_k-c_i)\sqrt{\beta}+n_k}{(c_k-c_i)\sqrt{\beta}+n_k-n_i}\\
+\sum_{i=1,i\neq k}^p n_i\,\dfrac{2\left(\beta+c_i\sqrt{\beta}+n_i\right)\left(2\beta+c_k\sqrt{\beta}+n_k+|\vec{n}|+1\right)}{\left(2\beta+c_i\sqrt{\beta}+n_i+|\vec{n}|\right)_2\left(( c_i-c_k)\sqrt{\beta}-n_k+n_i-1\right)}
\times\prod_{q=1,q\neq i}^p\dfrac{(c_i-c_q)\sqrt{\beta}+n_i}{(c_i-c_q)\sqrt{\beta}+n_i-n_q}
\end{multlined}\\
&\hspace{1cm}
 =
\begin{multlined}[t][.7\textwidth](n_k+1)\left(1+\dfrac{c_k\sqrt{\beta}+n_k-|\vec{n}|}{2\beta+c_k\sqrt{\beta}+n_k+|\vec{n}|+2}\right)\prod_{i=1,i\neq k}^p\left(1+\dfrac{n_i}{(c_k-c_i)\sqrt{\beta}+n_k+1-n_i}\right)\\
-{n_k\left(1+\dfrac{c_k\sqrt{\beta}+n_k-|\vec{n}|}{2\beta+c_k\sqrt{\beta}+n_k+|\vec{n}|}\right)}\prod_{i=1,i\neq k}^p\left(1+\dfrac{n_i}{(c_k-c_i)\sqrt{\beta}+n_k-n_i}\right)\\
+\sum_{i=1,i\neq k}^p n_i\,\dfrac{2\left(\beta+c_i\sqrt{\beta}+n_i\right)\left(2\beta+c_k\sqrt{\beta}+n_k+|\vec{n}|+1\right)}{\left(2\beta+c_i\sqrt{\beta}+n_i+|\vec{n}|\right)_2\left(( c_i-c_k)\sqrt{\beta}-n_k+n_i-1\right)}
\times\prod_{q=1,q\neq i}^p\dfrac{(c_i-c_q)\sqrt{\beta}+n_i}{(c_i-c_q)\sqrt{\beta}+n_i-n_q}.
\end{multlined}
\end{align*}

The components within the sum behave as follows 
when
 \( \beta \rightarrow \infty \):
\begin{align*}
\dfrac{n_i}{(c_i-c_k)\sqrt{\beta}},
\end{align*}
while the other two summands behave as follows for \( \beta \rightarrow \infty \), respectively
\begin{align*}
 &{(n_k+1)}
 \left(1+\dfrac{c_k}{2}\dfrac{1}{\sqrt{\beta}}-\dfrac{1}{\sqrt{\beta}}\sum_{i=1,i\neq k}^p\dfrac{n_i}{c_i-c_k}\right)+O\left(\dfrac{1}{\beta}\right), &
&n_k \left(1+\dfrac{c_k}{2}\dfrac{1}{\sqrt{\beta}}-\dfrac{1}{\sqrt{\beta}}\sum_{i=1,i\neq k}^p\dfrac{n_i}{c_i-c_k}\right)+O\left(\dfrac{1}{\beta}\right).
\end{align*}
Combining all the summands and simplifying, we find that the entire expression behaves as follows for \( \beta \rightarrow \infty \):
\begin{align*}
1+\dfrac{c_k}{2}\dfrac{1}{\sqrt{\beta}}
+O\left(\dfrac{1}{\beta}\right).
\end{align*}
The limit now becomes straightforward to apply, yielding \eqref{LaguerreSKRecurrence}.
\end{proof}

\begin{rem}
	An instance of Theorem \ref{teo:Hermite_recurrence} can be observed in \cite[Theorem 5]{ContinuosII} (also for $p$ weights) where a unit shift is applied to $n_1$. Consequently, there is no general permutation or arbitrary index involved. Moreover, the proof provided in \cite{ContinuosII} is solely for two weights.
\end{rem}

In the stepline, this expression transforms to:
\begin{cor}[Step line multiple Hermite recurrence coefficients]
The Hermite type I \eqref{HermiteTypeI} and type II \eqref{HermiteTypeII} multiple orthogonal polynomials in the step line adhere to a recurrence relation of the form \eqref{StepLineRecurrence} with respect to the coefficients
\begin{align*}
\begin{aligned}
	 b^0_{pm+k}&=\dfrac{c_{k+1}}{2}\\
b^j_{pm+k}&=\dfrac{1}{2^j}\sum_{i=1}^{k+1-j}(m+1)\prod_{q=k+2-j}^k{(c_i-c_q)}
+\dfrac{1}{2^j}\sum_{i=k+1}^{p}m\prod_{q=k+2-j}^k{(c_i-c_q)},&j&\in\{1,\dots,k\},\\
b^j_{pm+k}&=\dfrac{1}{2^j}\sum_{i=k+1}^{p+k+1-j}m\prod_{q=1}^k{(c_i-c_q)}\prod_{q=p+k+2-j}^p{(c_i-c_q)}, &j&\in\{k+1,\dots,p\}.
\end{aligned}
\end{align*}
\end{cor}

\section*{Conclusions and outlook}

In this paper, we present novel findings on explicit expressions for type I multiple Laguerre of the second kind orthogonal polynomials for an arbitrary number of weights. These expressions are represented in terms of multiple Kampé de Fériet series. Additionally, we derive explicit expressions for type I multiple orthogonal polynomials of multiple Hermite families with an arbitrary number of weights. Furthermore, we provide explicit expressions for nearest-neighbor and step line recursion coefficients for Jacobi–Piñeiro, Laguerre of the first and second kinds, and multiple Hermite polynomials. 

In future research, a pivotal goal lies in uncovering analogous hypergeometric expressions employing multiple Kampé de Fériet functions for Hahn multiple orthogonal polynomials with $p$ weights, along with all their discrete descendants within the Askey scheme \cite{AskeyII}. Such discoveries bear profound significance, especially in applications where type I polynomials hold sway, as seen in Markov chains characterized by transition matrices featuring $p$ subdiagonals. These findings open up fresh avenues for delving into the spectral properties of Markov chains endowed with $p$ banded Hessenberg transition matrices. The bidiagonal factorization discussed in \cite{bidiagonal} initially presented for the case of $p=2$ warrants extension to encompass scenarios where $p\geq 2$.

\section*{Acknowledgments}

AB acknowledges Centre for Mathematics of the University of Coimbra 
(funded by the Portuguese Government through FCT/MCTES, DOI: 10.54499/UIDB/00324/2020).

JEFD and AF
acknowledges the CIDMA Center for Research and Development in Mathematics and Applications
(University of Aveiro) and the Portuguese Foundation for Science and Technology (FCT) for their support within
projects, DOI: 10.54499/UIDB/04106/2020 \& DOI: 10.54499/UIDP/04106/2020; additionally, JEFD acknowledges the PhD contract DOI: 10.54499/UI/BD/152576/2022 from FCT Portugal.

MM acknowledges Spanish ``Agencia Estatal de Investigación'' research project [PID2021- 122154NB-I00], \emph{Ortogonalidad y Aproximación con Aplicaciones en Machine Learning y Teoría de la Probabilidad}.

\section*{Declarations}

\begin{enumerate}
	\item \textbf{Conflict of interest:} The authors declare no conflict of interest.
	\item \textbf{Ethical approval:} Not applicable.
	\item \textbf{Contributions:} All the authors have contribute equally.
\end{enumerate}



\begin{thebibliography}{99}
	
	\bibitem{afm}
	C. Álvarez-Fernández, U. Fidalgo Prieto, and M. Mañas,
	\emph{Multiple orthogonal polynomials of mixed type: Gauss--Borel factorization and the multi-component 2D Toda hierarchy}, Advances in Mathematics \textbf{227} (2011), 1451–1525. 
	
	\bibitem{andrews}
	G. E. Andrews, R. Askey, and R. Roy,
 \emph{Special Functions},
	Cambridge University Press, Cambridge, 1999.
	
	\bibitem{ContinuosII}
	A. I. Aptekarev, A. Branquinho, and W. Van Assche,
	\emph{Multiple orthogonal polynomials for classical weights}, Transactions of the American Mathematical Society \textbf{355} (10) (2003) 3887--3914.
	
	\bibitem{Arvesu}
	J. Arvesu, J. Coussement, and W. Van Assche,
	\emph{Some discrete multiple orthogonal polynomials},
	Journal of Computational and Applied Mathematics \textbf{153} (2003) 19--45.
	
	\bibitem{AskeyII}
	B. Beckerman, J. Coussement, and W. Van Assche,
	\emph{Multiple Wilson and Jacobi--Piñeiro polynomials},
	Journal of Approximation Theory \textbf{132} (2005) 155--181.
	
	\bibitem{HahnI}
	A. Branquinho, J. E. F. Díaz, A. Foulquié-Moreno, and M. Mañas,
	\emph{Hahn multiple orthogonal polynomials of type I: Hypergeometric expressions},
	Journal of Mathematical Analysis and Applications \textbf{528} (2023), 1277471.
	
\bibitem{bidiagonal}
A. Branquinho, J. E. F. Díaz, A. Foulquié-Moreno, and M. Mañas,
 \emph{Bidiagonal factorization of the recurrence matrix for the Hahn multiple orthogonal polynomials}, Linear Algebra and its Applications doi: \href{https://doi.org/10.1016/j.laa.2024.03.033}{10.1016/j.laa.2024.03.033}.

\bibitem{JPpmedidas}
A. Branquinho, J. E. F. Díaz, A. Foulquié-Moreno, and M. Mañas,
 \emph{Hypergeometric expressions for type I Jacobi-Piñeiro orthogonal polynomials with arbitrary number of weights}, to appear in Proceedings of the AMS. \hyperref{https://arxiv.org/abs/2310.18294}{}{}{arXiv:2310.18294v1}. 

	
	\bibitem{CRM} 
	A. Branquinho, J. E. F. Díaz, A. Foulquié-Moreno, and M. Mañas,  \emph{Markov Chains and Multiple Orthogonality,} To appear in ``Orthogonal Polynomials, Special Functions and Applications -- Proceedings of the 16th International Symposium, Montreal, Canada, In honor to Richard Askey'' in the CRM Series in Mathematical Physics, 2023. .
	
	\bibitem{finite} 
	A. Branquinho, J. E. F. Díaz, A Foulquié-Moreno, and Manuel Mañas,
	\emph{Finite Markov chains and multiple orthogonal polynomials}, 2023. \hyperref{https://arxiv.org/abs/2308.00182}{}{}{arXiv:2308.00182}. 
	
	\bibitem{hypergeometric} 
	A. Branquinho, J. E. F. Díaz, A. Foulquié-Moreno, and M. Mañas,
	\emph{Hypergeometric multiple orthogonal polynomials and random walks}, 2021. 
	\hyperref{https://arxiv.org/abs/2107.00770 }{}{}{arXiv:2107.00770 }.
	
	\bibitem{JP}
	A. Branquinho, J. E. F. Díaz, A. Foulquié-Moreno, M. Mañas, and C. Álvarez-Fernández,
	\emph{Jacobi--Piñeiro Markov chains}, to appear in Revista de la Real Academia de Ciencias Exactas, Físicas y Naturales. Serie A. Matemáticas \textbf{118} (15) (2024).
	
	\bibitem{aim} 
	A. Branquinho, A. Foulquié-Moreno, and M. Mañas, 
	\emph{Spectral theory for bounded banded matrices with positive bidiagonal factorization and mixed multiple orthogonal polynomials}, Advances in Mathematics \textbf{434} (2023) 109313. 
	
	\bibitem{laa} 
	A. Branquinho, A. Foulquié-Moreno, and M. Mañas, 
	\emph{Positive bidiagonal factorization of tetradiagonal Hessenberg matrices}, Linear Algebra and its Applications \textbf{667} (2023) 132--160. 
	
	\bibitem{phys-scrip} 
	A. Branquinho, A. Foulquié-Moreno, and M. Mañas, 
	\emph{Oscillatory banded Hessenberg matrices, multiple orthogonal polynomials and random walks}, Physica Scripta \textbf{98} (2023) 105223. 
	
	\bibitem{Contemporary}
	A. Branquinho, A. Foulquié-Moreno, and M. Mañas,
  \emph{Banded matrices and their orthogonality}, to appear in Recent Progress in Special Functions, Galina Filipuk, editor. AMS Contemporary Mathematics, 2024.
	
	\bibitem{Ismail}
	M. E. H. Ismail,
	\emph{Classical and Quantum Orthogonal Polynomials in One Variable},
	Cambridge University Press, Cambridge, 2005.
	
	
	
	
	
	
	
	
	\bibitem{nikishin_sorokin}
	E. M. Nikishin and V. N. Sorokin,
	\emph{Rational Approximations and Orthogonality},
	Translations of Mathematical Monographs \textbf{92},
	American Mathematical Society, Providence, 1991.
	
	\bibitem{slater}
	L. J. Slater,
	\emph{Generalized Hypergeometric Functions},
	Cambridge University Press, Cambridge, 2008.
	
	\bibitem{Srivastava}
	 H. M. Srivastava and P W. Karlsson, 
	 \emph{Multiple Gaussian Hypergeometric Series}, Ellis Horwood Limited, John Wiley \& Sons,
	Chichester, 1985.
	
	\bibitem{Clasicos}
	W. Van Assche and E. Coussement,
	\emph{Some classical multiple orthogonal polynomials},
	Journal of Computational and Applied Mathematics {\textbf{127}} (2001) 317–347.
	
	
%
%
%








 






















\end{thebibliography}
\end{document}